\documentclass[letterpaper,11pt]{article}
\usepackage[margin=1in]{geometry}  

\usepackage{bbm}
\usepackage{graphicx}
\usepackage{amsmath,amssymb,amsthm,amsfonts}

\usepackage{paralist}
\usepackage{bm}
\usepackage{xspace}
\usepackage{url}
\usepackage{prettyref}
\usepackage{boxedminipage}
\usepackage{wrapfig}
\usepackage{ifthen}
\usepackage{color}
\usepackage{xspace}

\usepackage{amsmath,amsthm,amsfonts,amssymb}
\usepackage{mathtools}
\usepackage{graphicx}

\usepackage{nicefrac}

\newtheorem*{definition*}{Definition}

\DeclareMathOperator*{\argmax}{argmax}
\DeclareMathOperator*{\argmin}{argmin}
\usepackage{subcaption}

\usepackage[utf8]{inputenc}

\usepackage{xcolor}
\definecolor{expert}{HTML}{008000}
\definecolor{error}{HTML}{f96565}

\usepackage{color-edits}
\addauthor{sw}{blue}

\usepackage{thmtools}
\usepackage{thm-restate}

\usepackage{tikz}
\usetikzlibrary{arrows,calc} 
\newcommand{\tikzAngleOfLine}{\tikz@AngleOfLine}
\def\tikz@AngleOfLine(#1)(#2)#3{%
\pgfmathanglebetweenpoints{%
\pgfpointanchor{#1}{center}}{%
\pgfpointanchor{#2}{center}}
\pgfmathsetmacro{#3}{\pgfmathresult}%
}

\declaretheoremstyle[
    headfont=\normalfont\bfseries, 
    bodyfont = \normalfont\itshape]{mystyle}

\usepackage[linesnumbered,algoruled,boxed,lined,noend]{algorithm2e}

\usepackage{listings}
\usepackage{amsmath}
\usepackage{amsthm}
\usepackage{tikz}
\usepackage{caption}
\usepackage{mdwmath}
\usepackage{multirow}
\usepackage{mdwtab}
\usepackage{eqparbox}
\usepackage{multicol}
\usepackage{amsfonts}
\usepackage{tikz}
\usepackage{multirow,bigstrut,threeparttable}
\usepackage{amsthm}
\usepackage{bbm}
\usepackage{epstopdf}
\usepackage{mdwmath}
\usepackage{mdwtab}
\usepackage{eqparbox}
\usetikzlibrary{topaths,calc}
\usepackage{latexsym}
\usepackage{cite}
\usepackage{amssymb}
\usepackage{bm}
\usepackage{amssymb}
\usepackage{graphicx}
\usepackage{mathrsfs}
\usepackage{epsfig}
\usepackage{psfrag}
\usepackage{setspace}
\usepackage[
            CJKbookmarks=true,
            bookmarksnumbered=true,
            bookmarksopen=true,
            colorlinks=true,
            citecolor=red,
            linkcolor=blue,
            anchorcolor=red,
            urlcolor=blue
            ]{hyperref}
\usepackage[linesnumbered,algoruled,boxed,lined]{algorithm2e}
\usepackage{algpseudocode}
\usepackage{stfloats}

\usepackage{comment}
\usepackage{mathtools}
\usepackage{blkarray}
\usepackage{multirow,bigdelim,dcolumn,booktabs}

\usepackage{xparse}
\usepackage{tikz}
\usetikzlibrary{calc}
\usetikzlibrary{decorations.pathreplacing,matrix,positioning}

\usepackage[T1]{fontenc}
\usepackage[utf8]{inputenc}
\usepackage{mathtools}
\usepackage{blkarray, bigstrut}
\usepackage{gauss}

\newcommand*{\BraceAmplitude}{0.4em}%
\newcommand*{\VerticalOffset}{0.5ex}%
\newcommand*{\HorizontalOffset}{0.0em}%
\newcommand*{\blocktextwid}{3.0cm}%
\NewDocumentCommand{\InsertLeftBrace}{%
	O{} 
	O{\HorizontalOffset,\VerticalOffset} 
	O{\blocktextwid} 
	m   
	m   
	m   
}{%
	\begin{tikzpicture}[overlay,remember picture]
	\coordinate (Brace Top)    at ($(#4.north) + (#2)$);
	\coordinate (Brace Bottom) at ($(#5.south) + (#2)$);
	\draw [decoration={brace, amplitude=\BraceAmplitude}, decorate, thick, draw=black, #1]
	(Brace Bottom) -- (Brace Top) 
	node [pos=0.5, anchor=east, align=left, text width=#3, color=black, xshift=\BraceAmplitude] {#6};
	\end{tikzpicture}%
}%
\NewDocumentCommand{\InsertRightBrace}{%
	O{} 
	O{\HorizontalOffset,\VerticalOffset} 
	O{\blocktextwid} 
	m   
	m   
	m   
}{%
	\begin{tikzpicture}[overlay,remember picture]
	\coordinate (Brace Top)    at ($(#4.north) + (#2)$);
	\coordinate (Brace Bottom) at ($(#5.south) + (#2)$);
	\draw [decoration={brace, amplitude=\BraceAmplitude}, decorate, thick, draw=black, #1]
	(Brace Top) -- (Brace Bottom) 
	node [pos=0.5, anchor=west, align=left, text width=#3, color=black, xshift=\BraceAmplitude] {#6};
	\end{tikzpicture}%
}%
\NewDocumentCommand{\InsertTopBrace}{%
	O{} 
	O{\HorizontalOffset,\VerticalOffset} 
	O{\blocktextwid} 
	m   
	m   
	m   
}{%
	\begin{tikzpicture}[overlay,remember picture]
	\coordinate (Brace Top)    at ($(#4.west) + (#2)$);
	\coordinate (Brace Bottom) at ($(#5.east) + (#2)$);
	\draw [decoration={brace, amplitude=\BraceAmplitude}, decorate, thick, draw=black, #1]
	(Brace Top) -- (Brace Bottom) 
	node [pos=0.5, anchor=south, align=left, text width=#3, color=black, xshift=\BraceAmplitude] {#6};
	\end{tikzpicture}%
}%

\usetikzlibrary{patterns}

\definecolor{cof}{RGB}{219,144,71}
\definecolor{pur}{RGB}{186,146,162}
\definecolor{greeo}{RGB}{91,173,69}
\definecolor{greet}{RGB}{52,111,72}

\theoremstyle{plain}
\newtheorem{theorem}{Theorem}

\newtheorem{lemma}{Lemma}
\newtheorem{remark}{Remark}
\newtheorem{corollary}{Corollary}

\def \bP {\mathbb{P}}
\def \bE {\mathbb{E}}
\def \bR {\mathbb{R}}

\def \cG {\mathcal{G}}

\def \cN {\mathcal{N}}

\def\1{\mathbbm{1}}

\usepackage{xspace}

\newcommand{\stepa}[1]{\overset{\rm (a)}{#1}}
\newcommand{\stepb}[1]{\overset{\rm (b)}{#1}}
\newcommand{\stepc}[1]{\overset{\rm (c)}{#1}}
\newcommand{\stepd}[1]{\overset{\rm (d)}{#1}}
\newcommand{\stepe}[1]{\overset{\rm (e)}{#1}}
\newcommand{\stepf}[1]{\overset{\rm (f)}{#1}}

\newcommand{\eps}{\varepsilon}

\newcommand{\naturals}{\mathbb{N}}

\newcommand{\lnorm}[2]{\left\|{#1} \right\|_{{#2}}}

\newcommand{\Fnorm}[1]{\lnorm{#1}{\rm F}}
\newcommand{\Opnorm}[1]{\lnorm{#1}{\rm op}}

\definecolor{myblue}{rgb}{.8, .8, 1}
\definecolor{mathblue}{rgb}{0.2472, 0.24, 0.6} 
\definecolor{mathred}{rgb}{0.6, 0.24, 0.442893}
\definecolor{mathyellow}{rgb}{0.6, 0.547014, 0.24}

\newcommand{\calA}{{\mathcal{A}}}
\newcommand{\calB}{{\mathcal{B}}}

\newcommand{\calE}{{\mathcal{E}}}

\newcommand{\calG}{{\mathcal{G}}}

\newcommand{\calN}{{\mathcal{N}}}

\newcommand{\calS}{{\mathcal{S}}}

\newcommand{\calX}{{\mathcal{X}}}

\newcommand{\jiao}[1]{\langle{#1}\rangle}

\usepackage{cleveref}
\crefname{lemma}{Lemma}{Lemmas}
\Crefname{lemma}{Lemma}{Lemmas}
\crefname{thm}{Theorem}{Theorems}
\Crefname{thm}{Theorem}{Theorems}
\Crefname{assumption}{Assumption}{Assumptions}
\Crefname{inftheorem}{Informal Theorem}{Informal Theorems}
\crefformat{equation}{(#2#1#3)}

\newcommand{\trace}{{\mathsf{Tr}}}
\newcommand{\GW}{{\mathsf{GW}}}
\newcommand{\QMLE}{{\mathsf{QMLE}}}
\newcommand{\id}{{\text{Id}}}
\newcommand{\hS}{\widehat{\Sigma}}

\def\simiid{\stackrel{iid}{\sim}}

\begin{document}

\title{Covariance alignment: from maximum likelihood estimation to Gromov--Wasserstein}

\author{
Yanjun Han\thanks{Courant Institute of Mathematical Sciences and Center for Data Science, New York University, New York NY, United States}
\and
Philippe Rigollet\thanks{Institute for Data, Systems, and Society, Massachusetts Institute of Technology, Cambridge MA, United States} \thanks{Department of Mathematics, Massachusetts Institute of Technology, Cambridge MA, United States}
\and
George Stepaniants\footnotemark[2] \footnotemark[3]
}

\maketitle

\begin{abstract}
Feature alignment methods are used in many scientific disciplines for data pooling, annotation, and comparison. As an instance of a permutation learning problem, feature alignment presents significant statistical and computational challenges. In this work, we propose the \textit{covariance alignment} model to study and compare  various  alignment methods and establish a minimax lower bound for covariance alignment that has a non-standard dimension scaling because of the presence of a nuisance parameter. This lower bound is in fact minimax optimal and is achieved by a natural quasi MLE. However, this estimator involves a search over all permutations which is  computationally infeasible even when the problem has moderate size. To overcome this limitation, we show that the celebrated Gromov--Wasserstein algorithm from optimal transport which is more amenable to fast implementation even on large-scale problems is also minimax optimal. These results give the first statistical justification for the deployment of the Gromov--Wasserstein algorithm in practice.
\end{abstract}

\tableofcontents

\section{Introduction}
Modern research is marked by an unprecedented access to vast computational, data collection, and storage resources. Standing in for hand-crafted and expensive statistical experiments, a deluge of data is available at a scale on which even straightforward statistical techniques have the potential to yield remarkably nontrivial insights. This paradigm shift has transformed countless domains in science, engineering, and management science, to name a few, and also brought a new set of challenges including the integration, or \emph{alignment}, of several datasets collected from different sources.
This work proposes a statistical framework to study the fundamental limits associated with this phenomenon and establishes the optimality of state-of-the-art alignment methods based on optimal transport.

Motivated by applications to biostatistics, we formulate the feature alignment problem as follows. Consider two datasets  that consist of $m$ and $n$ observations respectively. Each dataset can be represented as a features-by-observations matrix:
\begin{align}\label{eq:datasets2}
    X^m ={ \Big[X_1, \hdots, X_m\Big] } \in \bR^{d \times m}, \qquad Y^n = { \Big[Y_1, \hdots, Y_n\Big] } \in \bR^{d \times n}.
\end{align}
The goal of feature alignment is to find a one-to-one  mapping (permutation) $\pi: S \to [d]$ from a subset $S \subseteq [d]$ of features in $X^m$ to a subset $\pi(S) \subseteq [d]$ of features in $Y^n$.

This setup arises for example in untargeted metabolomics, proteomics, and lipidomics studies which record the concentrations of $d$ compounds (metabolites, proteins, and lipids) across a collection of $n$ patients to find biomarker compounds that are associated with the health status of a patient. The term ``untargeted'' refers to the fact that each of the $d$ features corresponds to the concentration level of an unnamed extracted compound measured across all patients in the study. Matching features (compounds) between two such studies is a necessary step for transferring annotated labels between studies, comparing studies to discover shared compounds, and merging multiple studies to increase their sample size and ultimately increase the statistical power of downstream inference tasks.

To investigate this question, we propose and analyze a novel statistical model where the observations are centered Gaussian vectors in the idealized case where there exists a matching of all $d$ features ($S=[d]$). More specifically, let $\Sigma \succ 0$ be an unknown covariance matrix and let $\pi^\star$ be an unknown permutation of $[d]$. We assume that 
$$
X_1, \ldots, X_m \simiid \cN(0, \Sigma)\quad \text{and} \quad Y_1, \ldots, Y_n \simiid \cN(0, \Sigma^{\pi^\star})
$$
where $(\Sigma^{\pi^\star})_{i,j} := \Sigma_{\pi^\star(i), \pi^\star(j)}$ and the two samples are assumed to be independent. It is not hard to see that in this context the sample covariance matrices $\hS_X = X^m(X^m)^\top / m$ and $\hS_Y = Y^n(Y^n)^\top / n$  associated to each of the two samples are sufficient statistics for $\pi^\star$ and the problem boils down to aligning these covariance matrices. We call this problem \emph{covariance alignment}.

Treating the covariance matrix $\Sigma$ as a nuisance parameter and $\pi^\star$ as the parameter of interest, one may write a profile log-likelihood for $\pi$---see~\eqref{eq:loglik}---which can be further simplified into a quasi log-likelihood that is quadratic in $\pi$---see~\eqref{eq:quasiloglik}. Our first estimator is obtained by maximizing this quasi log-likelihood and is thus called the \emph{quasi maximum likelihood estimator} (QMLE in short):
\begin{align}\label{eq:QMLE_intro}
    \widehat \pi^\QMLE = \argmin_{\pi \in S_d}\langle P_\pi\widehat{\Sigma}_X^{-1}P_\pi^\top, \widehat{\Sigma}_Y\rangle,
\end{align}
where $S_d$ denotes the set of all permutations and $P_\pi$ is the permutation matrix associated to $\pi$. 
One contribution of the paper is that the QMLE is optimal in a minimax sense. Unfortunately, the statistical optimality of the QMLE is contrasted by its associated computational burden. Indeed, the QMLE is a quadratic assignment problem that requires optimization over the set $S_d$ of $d!$ permutations, which is not practically feasible, even for moderate values of $d$. To overcome this situation, we borrow from computational optimal transport~\cite{peyre2017computational} and propose a Gromov--Wasserstein (GW) estimator which optimizes a similar objective over the Birkhoff polytope $\text{BP}(d)$ defined in~\eqref{eq:birkhoff_polytope} as the convex hull of the set of permutation matrices:
\begin{equation}
    \label{eq:GWopt}
    \widehat{P}^\GW = \argmax_{P \in \text{BP}(d)}\langle P\widehat{\Sigma}_XP^\top, \widehat{\Sigma}_Y\rangle.
\end{equation}
We show in Section~\ref{sec:estimator} that  $\widehat{P}^\GW=P_{\widehat \pi^\GW}$ is in fact a permutation matrix corresponding to a permutation $\widehat \pi^\GW$ so that this relaxation is tight. Despite the structural similarity between the objective functions in the previous two displays,  we note in Section~\ref{sec:estimator} that the same relaxation for the QMLE is not tight; in other words, the solution to
$$
\argmin_{P \in \text{BP}(d)}\langle P\widehat{\Sigma}_X^{-1}P^\top, \widehat{\Sigma}_Y\rangle,
$$
is not achieved at a permutation in general, but rather in the interior of the Birkhoff polytope. While this observation does not preclude the existence of good rounding schemes, we do not pursue this route. Instead, we show that the GW estimator is also minimax optimal. 

In light of its advantageous computational properties, this work provides theoretical support in favor of the Gromov--Wasserstein matcher recently proposed in a metabolomic study~\cite{breeur2023optimal}.

\begin{remark}
    While the Birkhoff polytope is tractable in the sense that it is convex and described by $O(n^2)$ inequalities, we still have to resort to a heuristic to solve~\eqref{eq:GWopt} because the objective is nonconvex. In fact, it is a convex \emph{maximization} problem and a classical maximum principle indicates that the solution is achieved at a vertex of the Birkhoff polytope, namely a permutation. Despite this lack of convexity, the objective function is smooth and lends itself to the gradient-based heuristic in~\cite{SolPeyKim16} and is successfully deployed in a companion applied work~\cite{breeur2023optimal}.
\end{remark}

\subsection{Main results}

We first recall the covariance alignment model. Let $\Sigma \succ 0$ be an unknown covariance matrix and let $\pi^\star$ be an unknown permutation of $[d]$. Assume that we observe
$$
X_1, \ldots, X_m \simiid \cN(0, \Sigma)\quad \text{and} \quad Y_1, \ldots, Y_n \simiid \cN(0, \Sigma^{\pi^\star})
$$
where $(\Sigma^{\pi^\star})_{i,j} := \Sigma_{\pi^\star(i), \pi^\star(j)}$ and the two samples are assumed to be independent. 

Our goal is to estimate the permutation $\pi^\star$. While $\Sigma$ is a nuisance parameter, it does play an important role in our ability to identify $\pi^\star$. For example, if $\Sigma$ is the identity or has two identical columns, then $\pi^\star$ is not identifiable. From this simple observation, it is clear that the performance of any estimator will be greatly influenced by the structure of $\Sigma$. Rather than imposing structural conditions on $\Sigma$, we propose to evaluate our estimation using a  $\Sigma$-dependent (pseudo-)distance, following an approach similar to the one employed in~\cite{flammarion2019optimal} for statistical seriation. More specifically, for any two permutations  $\pi, \pi' \in S_d$, define
$$
d_\Sigma(\pi, \pi'):=\| \Sigma^{\pi}-\Sigma^{\pi'}\|_{\text{\rm F}}.
$$
It follows readily from the triangle inequality that $d_\Sigma$ is a pseudo-distance for any $\Sigma \succeq 0$. It is less sensitive to errors that consist in permuting two rows/columns that are close in Euclidean norm. More discussions on the loss function can be found in Section \ref{subsec:other_minimax}.

Our main result is the following characterization of the optimal sample complexity for estimating the permutation $\pi$ under the Frobenius loss. 

\begin{theorem}\label{thm:main}
For $\Opnorm{\Sigma}\le 1$ and $\varepsilon\in (0,1)$, both the QMLE $\widehat \pi^\QMLE$~\eqref{eq:QMLE_intro} and the GW estimator $\widehat{\pi}^\GW$~\eqref{eq:GWopt}
achieve $\bE\| \Sigma^{\widehat{\pi}}-\Sigma^{\pi^\star}\|_{\text{\rm F}} \le \varepsilon$ as long as
\begin{align*}
    \min\{m,n\} \ge C\frac{d\log d}{\varepsilon^2} \quad \text{ and } \quad mn\ge C\frac{d^3\log d}{\varepsilon^4}.
\end{align*}
Here $C>0$ is a numerical constant independent of $(d,\varepsilon)$. Up to multiplicative constants, the above conditions on the sample complexity are not improvable in general. 
\end{theorem}

The results on the minimax rate of convergence, as well as the high probability upper bounds on the Frobenius loss, are deferred to Theorems \ref{thm:upper_bound_QMLE}-\ref{thm:lower_bound}. Some remarks of Theorem \ref{thm:main} are in order: 

\medskip
\noindent (i) $\min\{m,n\}= \Omega((d\log d)/\varepsilon^2)$ ensures that based solely on $X^m$ or $Y^n$, there is an estimator $\widehat{\pi}$ with $\bE\| \Sigma^{\widehat{\pi}}-\Sigma^{\pi^\star}\|_{\text{\rm F}} \le \varepsilon$ \emph{when $\Sigma$ is known}. In other words, this is the sample complexity of permutation estimation with a known covariance. 

\medskip

\noindent (ii)  $mn= \Omega((d^3\log d)/\varepsilon^4)$  accounts for the effect of an unknown covariance $\Sigma$, and interpolates between two extreme scenarios. First, if $m=\Omega(d^2/\varepsilon^2)$ and $n=\Omega((d\log d)/\varepsilon^2)$, one may estimate $\Sigma$ within Frobenius loss $\varepsilon$ solely based on $X^m$, and then estimate $\pi^\star$ based on $Y^n$ as if $\Sigma$ were known---see (i) above. A more illuminating scenario arises if $m=n$. In that case,  the above condition reduces to $m=n=\Omega( \sqrt{d^3\log d}/\varepsilon^2)$. This peculiar sample complexity is in sharp contrast with the $\Theta(d^2/\varepsilon^2)$ sample size required to estimate the covariance matrix $\Sigma$ within accuracy $\eps$; in other words, the nuisance is allowed to be estimated at a slower rate without hurting the estimation of target parameters. This phenomenon is reminiscent of high-dimensional semiparametric statistics~\cite{chernozhukov2017double}.

\medskip
\noindent (iii)  $\Opnorm{\Sigma}\le 1$ ensures that the hard problem instance appears in the high-dimensional scenario. If it is replaced by the Frobenius norm, the hard instance will become low-dimensional and we no longer enjoy the $\text{poly}(d)$ reduction in the sample complexity compared with covariance estimation. This issue could be further mitigated by considering a rescaled Frobenius loss; we discuss the details in Section \ref{subsec:other_minimax}.

\subsection{Related work}
The estimation of permutations has become a preponderant challenge in various problems. It brings singular challenges, both from a statistical and computational perspective. In the rest of this section, we review some work in this growing area and comment on its connection with our setup and contributions.

\paragraph*{Dataset alignment and biological applications} Dataset alignment is a primitive problem in statistics and machine learning, used broadly in computer vision for matching local descriptors in images~\cite{lowe2004distinctive, bay2006surf, SolPeyKim16}, machine translation for converting words between languages~\cite{AlaGraCut18, alvarez2018gromov, alvarez2019towards, GraJouBer19}, and biostatistics for matching untargeted compounds, transferring labels, and merging unlabeled datasets~\cite{lin2013combinatorial, habra2021metabcombiner, climaco2022finding, breeur2023optimal}. The alignment problem is to find a correspondence between two $d$-dimensional datasets $X^m = [X_1, \hdots, X_m] \in \bR^{d \times m}, Y^n = [Y_1, \hdots, Y_n] \in \bR^{d \times n}$ either by matching their features (rows), matching their observations (columns), or both.

There are several biological applications where the features (rows) of $X^m, Y^n$ are unordered or misaligned, and hence need to be matched. These applications include untargeted metabolomics, proteomics, and lipidomics experiments where unnamed features (compounds such as metabolites, proteins, and lipids) are collected and their concentrations are measured in a set of patients (observations). Untargeted studies are common practice in biology as they can be simpler to perform than targeted studies with labeled features and can discover new biomarker compounds that enable drug discovery and tracking of disease progression~\cite{pirhaji2016revealing, wishart2019metabolomics, loftfield2021novel}. Recent works in metabolomics and proteomics have explored the applications of matching algorithms through combinatorial search, nearest neighbor methods, and optimal transport to solve the untargeted feature alignment problem~\cite{lin2013combinatorial, habra2021metabcombiner, climaco2022finding, breeur2023optimal}.

\paragraph*{Feature alignment and linear assignment} In many machine learning applications, data features have a predefined ordering and the aim is to match the observations in $X^m$ and $Y^n$. Examples include matching a set of patches in two images or word embeddings from two different languages. A classical approach is to form a similarity matrix $A_{ij} = a(X_i, Y_j)$ and maximize the objective
\begin{align}\label{eq:linear_assignment}
    \max_{\pi}\sum_{i=1}^m\sum_{j=1}^n A_{i, \pi(i)}
\end{align}
known as the \textit{(linear) assignment problem}. This combinatorial optimization problem \eqref{eq:linear_assignment} can be solved exactly in polynomial time using linear programming and the Hungarian algorithm~\cite{munkres1957algorithms}, or through optimal transport by lifting the matchings to the space of probabilistic couplings~\cite{peyre2017computational} where the Sinkhorn algorithm has allowed researchers to to process this problem in near-linear time~\cite{Cut13, AltWeeRig17}. A statistical model for the assignment problem was proposed in~\cite{collier2016minimax, galstyan2022optimal} where it is assumed that $X^m$ and $Y^n$ share a subset of observations corrupted with additive noise. When $X^m$ and $Y^n$ are the set of vertices in a bipartite random graph, the sharp statistical properties of \eqref{eq:linear_assignment} have been established in \cite{ding2023planted} for a planted matching problem. 

\paragraph*{Graph matching and quadratic assignment} Another fundamental instantiation of dataset alignment is through \emph{graph matching}, which is closest to our covariance alignment model. Based on $X^m$ and $Y^n$,  form two similarity matrices $A,B\in \mathbb{R}^{d\times d}$ over the features, treat them as weighted undirected graphs, and aim to find a permutation $\pi\in S_d$ such that $A^\pi\approx B$. This is usually done by solving the following combinatorial optimization problem
\begin{equation}\label{eq:quadratic_assignment}
    \max_{\pi \in S_d} \sum_{i=1}^d\sum_{j=1}^d A_{\pi(i)\pi(j)}B_{ij}, 
\end{equation}
This is an example of the \textit{quadratic assignment problem} (QAP)~\cite{koopmans1957assignment}, ``one of the most difficult problems in the NP-hard class''~\cite{LoiAbrBoa07}. This tantalizing difficulty has driven researchers to explore specific instantiations of QAP that are potentially easier to solve. Perhaps the most famous one is graph isomorphism which is strongly believed to not be NP-complete.
Another route to solve this problem in polynomial time is to consider random instances of QAP. Most noticeably, a recent line of research culminating with~\cite{mao2023random,MaoRudTik23} has investigated the alignment of correlated Erd\"os-R\'enyi random graphs from both the statistical and computational perspective.

Although the QAP in \eqref{eq:quadratic_assignment} has appeared in the graph matching literature \cite{hall2023partial, wu2022settling}, our covariance alignment model differs from the above graph matching papers in two aspects. First, under our Gaussian covariance model, the matrices $\hS_X$ and $\hS_Y$ being aligned are Wishart random graphs, departing from the traditional Erd\"os-R\'enyi, Wigner, or random geometric graphs. Second, and more important, all above works assume a \emph{known} distribution of the random matrices $(A, B)$. In contrast, our covariance alignment model treats the true covariance $\Sigma$ as an additional nuisance parameter, and we need to estimate the marginal distributions of $\hS_X$ and $\hS_Y$. To the best of our knowledge, the statistical problem of estimating $\pi$ in this setting has not been explored.

\paragraph*{Gromov--Wasserstein distance} The Gromov-Hausdorff distance~\cite{gromov1999metric} from metric geometry gives us yet another perspective on the quadratic assignment objective \eqref{eq:quadratic_assignment}. Consider two finite metric spaces given by their $d\times d$ pairwise distance metrics, $A$ and $B$ respectively. The Gromov-Hausdorff distance between these two metric spaces can be defined as
$$
d_{\sf GH}(A,B):=\min_{\pi \in S_d}\max_{1\le i,j \le d} \big| A_{ij} - B_{\pi(i)\pi(j)}\big|\,.
$$
Using an idea from optimal transport, M\'emoli~\cite{memoli2011gromov} proposed the Gromov--Wasserstein distance using  a double\footnote{The sup norm is relaxed to the squared Euclidean norm and the set of permutations is relaxed to the Birkhoff polytope.} relaxation:
$$
  d^2_{\sf GW}(A,B):=  \min_{P \in \text{BP}(d)} \sum_{i=1}^d\sum_{j=1}^d\sum_{k=1}^d\sum_{l=1}^dP_{ij}P_{kl}(B_{ik} - A_{jl})^2\,.
$$
Note that the above optimization problem is equivalent to
\begin{equation}\label{eq:gromov_wasserstein}
\max_{P \in \text{BP}(d)} \langle P A P^\top, B\rangle.
\end{equation}
The above formulation \eqref{eq:gromov_wasserstein} is equivalent to the quadratic assignment \eqref{eq:quadratic_assignment} for PSD $A, B \succeq 0$ and has received significant attention from the optimization community with various algorithms ranging from Frank-Wolfe~\cite{ferradans2014regularized, flamary2021pot} to alternating linear programs~\cite{memoli2011gromov} and alternating Sinkhorn updates with entropic regularization~\cite{SolPeyKim16, sejourne2021unbalanced}. Some progress has been made in studying the convergence of GW optimization~\cite{li2022fast, li2023convergent} and exact recovery of matchings for Bernoulli graphs~\cite{lyzinski2015graph}. In parallel, the statistical community has also studied the statistical properties of the Gromov--Wasserstein distance; see~\cite{zhang2022gromov} and references therein. The setup of this line of work is quite different from the one pursued in this paper: it consists in assuming that $d$ points used to form the distance matrices $A$ and $B$ are sampled i.i.d. from two metric measure spaces $\calA$ and $\calB$ respectively, and the goal is to study the convergence of $d_{\sf GH}(A,B)$ to $d_{\sf GH}(\calA,\calB)$ as $d \to \infty$.

The precise objective \eqref{eq:gromov_wasserstein} studied here has been applied to match feature covariance between untargeted metabolomic datasets in~\cite{breeur2023optimal}, while its statistical properties are underexplored. Our work presents the first statistical rates of estimation for GW in learning the matching between two datasets by aligning their covariance matrices.

\paragraph*{Additional literature on permutation estimation} Finally, we remark that our work falls in the broader line of work associated to permutation estimation. While more directly related to the graph matching problem mentioned above, our technical derivations present similarities with other statistical problems parameterized by a permutation. These include statistical seriation~\cite{flammarion2019optimal}, regression with shuffled data~\cite{pananjady2017linear}, isotonic regression~\cite{pananjady2022isotonic}, and noisy sorting~\cite{mao2018minimax}, to name a few.

\subsection{Notation} 
We use the following notations throughout the paper. 

\paragraph*{Scalars, vectors and matrices} For $d \in \naturals$, let $\1_d$ denote the $d$-length vector of all ones. For $n\in \naturals$, let $[n]=\{1,\cdots,n\}$. Let $S_n$ be the set of all permutations over $[n]$. For $\pi \in S_n$ and $A\in \bR^{n\times n}$, let $A^\pi$ be the $n\times n$ matrix with $(A^\pi)_{i,j} = A_{\pi(i), \pi(j)}$. Occasionally we also write $A^\pi$ in equivalent matrix forms as $A^\pi = P_\pi A P_\pi^\top$, where $P_\pi$ is the permutation matrix associated with $\pi$, formally defined as $(P_{\pi})_{i,j} = \1(\pi(i)=j)$. For matrices $A$ and $B$ of the same size, let $\jiao{A,B} = \trace(AB^\top)$ be their matrix inner product. For a matrix $A$, let $\Fnorm{A} = \sqrt{\jiao{A,A}}$ be its Frobenius norm, and $\Opnorm{A} = \max_{\|v\|_2=1} \|Av\|_2$ be its operator norm.  Finally for two real numbers $a$ and $b$, we use the shorthand notation $a\wedge b:=\min(a,b),\ a\vee b=\max(a,b)$.

\paragraph*{Information-theoretic quantities} For probability measures $P$ and $Q$ over the same probability space, let $D_{\text{KL}}(P\|Q)=\int \text{d}P\log(\text{d}P/\text{d}Q)$ and $\chi^2(P\|Q) = \int (\text{d}P)^2/\text{d}Q - 1$ be the Kullback-Leibler (KL) divergence and $\chi^2$ divergence between $P$ and $Q$, respectively. For a random vector $(X,Y,Z)\sim P_{XYZ}$, let $I(X;Y\mid Z)=\bE_{P_Z}[D_{\text{KL}}(P_{XY\mid Z}\|P_{X\mid Z}\otimes P_{Y\mid Z})]$ be the conditional mutual information between $X$ and $Y$ given $Z$. 

\paragraph*{Asymptotics} For non-negative sequences $\{a_n\}$ and $\{b_n\}$, we write $a_n =O( b_n )$ whenever there exists a numerical constant $C<\infty$ such that $a_n \le Cb_n$ for all $n$. Similarly, we write $a_n = \Omega(b_n)$ if $b_n = O(a_n)$, and $a_n = \Theta(b_n)$ or $a_n\asymp b_n$ if both $a_n = O(b_n)$ and $b_n = O(a_n)$. 

\section{Covariance alignment estimators}\label{sec:estimator}
In this section we investigate the statistical performance of two estimators: the quasi maximum likelihood estimator (QMLE) and the Gromov--Wasserstein (GW) estimator. In Section \ref{subsec:estimator}, we introduce both the QMLE and the GW estimator. In Section \ref{subsec:minimax}, we establish the minimax rate-optimality of both estimators. Despite a similar statistical performance, Section \ref{subsec:numerical} discusses the computational benefits of the GW estimator over the QMLE, supported by numerical experiments. 

\subsection{Estimators}\label{subsec:estimator}
We introduce the QMLE and the GW estimator separately.

\subsubsection{Quasi Maximum Likelihood Estimator}
The log-likelihood in the covariance alignment model is given, up to an additive constant, by
\begin{align*}
    \ell(\pi, \Sigma) &= \log\det(\Sigma^{-1}) - \frac{1}{m+n}\sum_{i=1}^{m}X_i^\top\Sigma^{-1}X_i - \frac{1}{m+n}\sum_{i=1}^{n}(P_{\pi^{-1}}Y_i)^\top\Sigma^{-1}(P_{\pi^{-1}}Y_i)\,,
\end{align*}
which is concave in $\Sigma^{-1}$ and maximized at
$$
\widehat \Sigma= \frac{m}{m + n}\hS_X + \frac{n}{m + n}(\hS_Y)^{(\pi^{-1})}, 
$$
leading to the (profile) log-likelihood
\begin{align}
\ell(\pi) := \sup_{\Sigma} \ell(\pi, \Sigma) = -\log\det\Big(\frac{m}{m + n}\hS_X + \frac{n}{m + n}(\hS_Y)^{(\pi^{-1})}\Big) - 1. \label{eq:loglik}
\end{align}
Due to the nonlinearity of the $\log\det$ function, properties of the maximum likelihood estimator obtained by maximizing~\eqref{eq:loglik} are difficult to establish. Instead, consider a quasi-log-likelihood obtained as follows. Note that for that $\pi\approx \pi^\star$, we have
$$
\frac{m}{m + n}\hS_X + \frac{n}{m + n}(\hS_Y)^{(\pi^{-1})}\approx \hS_X\,.
$$
A first-order Taylor expansion of $A\mapsto \log\det(A)$ around $\hS_X$, yields
$$
\log\det (A) \approx \langle \hS_X^{-1}, A-\hS_X\rangle = \langle \hS_X^{-1}, A\rangle + \text{const.}
$$
Hence
$$
\ell(\pi) \approx -\frac{n}{m+n}\langle  \hS_X^{-1}, (\hS_Y)^{(\pi^{-1})}\rangle + \text{const.}
$$
Maximizing this approximation is equivalent to maximizing the quasi log-likelihood $\widetilde \ell$ defined by
\begin{equation}
    \label{eq:quasiloglik}
 \widetilde \ell(\pi):=   \jiao{\hS_Y, (\hS_X^{\pi})^{-1}}\,.
\end{equation}
Consequently, our final quasi MLE $\widehat{\pi}^\QMLE$ is given by
\begin{align}\label{eq:QMLE}
\widehat{\pi}^\QMLE = \argmin_{\pi\in S_d} R_{m, n}^\QMLE(\pi) := \jiao{\hS_Y, (\hS_X^{\pi})^{-1}}. 
\end{align}
This quasi MLE can also be viewed as a two-step estimator where we first estimate the covariance as $\Sigma \approx \hS_X$ and then plug it into the log-likelihood for $\pi$.

\subsubsection{Gromov--Wasserstein Estimator}
The GW estimator aims to solve the quadratic assignment problem \eqref{eq:quadratic_assignment} in random graph matching, with similarity matrices $A,B$ given by the sample covariance matrices $\hS_X, \hS_Y$. Formally, the GW estimator $\widehat{\pi}^\GW$ is given by
\begin{align}\label{eq:GW_estimator}
    \widehat{\pi}^\GW &= \argmax_{\pi\in S_d} \jiao{\hS_X^\pi, \hS_Y} = \argmin_{\pi \in S_d} R_{m, n}^\GW(\pi) := \Fnorm{ \hS_X^{\pi} - \hS_Y }^2. 
\end{align}
The name GW comes from the mathematical theory of optimal transport which lifts the problem of permutation estimation to probabilistic couplings. Specifically, the GW optimization is performed as
\begin{align*}
    \widehat{P}^\GW = \argmax_{P \in \text{BP}(d)}\langle P\widehat{\Sigma}_XP^\top, \widehat{\Sigma}_Y\rangle, 
\end{align*}
where the maximization is over the Birkhoff polytope given by
\begin{equation}\label{eq:birkhoff_polytope}
    \text{BP}(d) = \{P \in \mathbb{R}_+^{d \times d}: P\1_d = P^\top\1_d = \1_d\}.
\end{equation}
Since this is a convex cost in $P$, the maximizer $\widehat{P}^\GW$ is one of the vertices of the Birkhoff polytope which are exactly the set of permutation matrices $S_d$. Therefore, the relaxation to the Birkhoff polytope is tight, while allowing for a broad diversity of continuous optimization algorithms unlike the combinatorial optimization in \eqref{eq:GW_estimator}.~\\

Unlike the GW estimator, the objective for the quasi MLE defined over the set of permutations cannot in general be relaxed to the space of coupling matrices. In fact, even for the limiting case $\hS_X = \hS_Y = \Sigma$ with $m,n = \infty$, von Neumann's trace inequality gives 
\begin{align*}
\min_{\pi\in S_d} \jiao{\Sigma, (\Sigma^{\pi})^{-1}} \ge d. 
\end{align*}
However, for the choice $P = \frac{1}{d}\1_d\1_d^\top \in \text{BP}(d)$ of the coupling matrix, it holds that
\begin{align*}
\jiao{ P\Sigma^{-1}P^\top, \Sigma } = \frac{\1_d^\top \Sigma^{-1} \1_d}{d}\cdot \frac{\1_d^\top \Sigma \1_d}{d},  
\end{align*}
which could be smaller than $d$ for many covariance matrices $\Sigma$ (e.g. when $\Sigma = \text{diag}(1,1/2,\cdots,1/2)$). This means that the minimizing coupling matrix may lie in the interior of the Birkhoff polytope.  

\subsection{Minimax optimality}\label{subsec:minimax}
Here we state the main results showing the minimax optimality of the quasi MLE and GW estimators.
\begin{theorem}[Upper bound for QMLE]\label{thm:upper_bound_QMLE}
Fix $\delta\in (0,1)$. Assume that 
$$
m\wedge n\ge c_0(d\log d+\log(1/\delta))\quad \text{and} \quad mn\ge c_0d(d\log d+\log(1/\delta))(d+\log(1/\delta))
$$
for a large enough numerical constant $c_0>0$. Then for every $(\pi^\star, \Sigma)$ with $\Opnorm{\Sigma}\le 1$, with probability at least $1-\delta$ it holds that
\begin{align*}
    \Fnorm{\Sigma^{\widehat{\pi}^\QMLE} - \Sigma^{\pi^\star}}^2 \lesssim  \frac{d\log d + \log(1/\delta)}{m\wedge n} + \sqrt{\frac{d(d\log d+\log(1/\delta))(d+\log(1/\delta))}{mn}} . 
\end{align*}
\end{theorem}

\begin{theorem}[Upper bound for GW]\label{thm:upper_bound_GW}
Fix $\delta\in (0,1)$, and assume that 
$$
m \wedge n \ge d\log d+\log(1/\delta).
$$
Then for every $(\pi^\star, \Sigma)$ with $\Opnorm{\Sigma}\le 1$, with probability at least $1-\delta$ it holds that
\begin{align*}
    \Fnorm{\Sigma^{\widehat{\pi}^\GW} - \Sigma^{\pi^\star}}^2 \lesssim \frac{d\log d + \log(1/\delta)}{m\wedge n} + \sqrt{\frac{d(d\log d+\log(1/\delta))(d+\log(1/\delta))}{mn}}.
\end{align*}
\end{theorem}

Integrating the tails over $\delta\in (0,1)$ we readily obtain the sample complexity upper bounds in Section \ref{thm:main}. Note that Theorem \ref{thm:upper_bound_QMLE} holds under an additional condition on the product $mn$ which does not appear in Theorem \ref{thm:upper_bound_GW}; we believe this extra condition to be superfluous. It comes from a technicality in our analysis that would be overcome using a concentration result for linear functionals of inverse-Wishart random matrices. In absence of such a result, we perform a local Taylor approximation instead. On the other hand, this condition remains necessary for either bound to be meaningful so we do not pursue this question here.

The following minimax lower bound shows that the rates of convergence in Theorem \ref{thm:upper_bound_QMLE} and \ref{thm:upper_bound_GW} are tight.

\begin{theorem}[Minimax lower bound]\label{thm:lower_bound}
Given observations $X^m\sim \cN(0,\Sigma)^{\otimes m}$ and $Y^n\sim \calN(0,\Sigma^{\pi^\star})^{\otimes n}$, the following minimax lower bound holds: if $m\wedge n\ge \log d$ and $m\vee n \ge d$, there exists an absolute constant $c_0>0$ independent of $(m,n,d)$ such that
\begin{align*}
\inf_{\widehat{\pi}}\sup_{\pi^\star, \Sigma: \Opnorm{\Sigma} \le 1} \mathbb{E}_{\Sigma} \Fnorm{\Sigma^{\widehat{\pi}} - \Sigma^{\pi^\star}}^2  \ge c_0\left(\frac{d\log d}{m\wedge n} + \sqrt{\frac{d^3\log d}{mn}}\right). 
\end{align*}
\end{theorem}

\noindent We note that the conditions $m\vee n\ge d$ and $m\wedge n\ge \log d$ in Theorem \ref{thm:lower_bound} are not entirely superfluous: they ensure that the minimax lower bound is no larger than $\Theta(d)$, as $\|\Sigma^{\widehat{\pi}} - \Sigma^{\pi^\star}\|_{\text{F}}^2 = O(d)$ trivially holds for every $\widehat{\pi}$.

\begin{figure}[t]
    \centering
    \includegraphics[width=\textwidth]{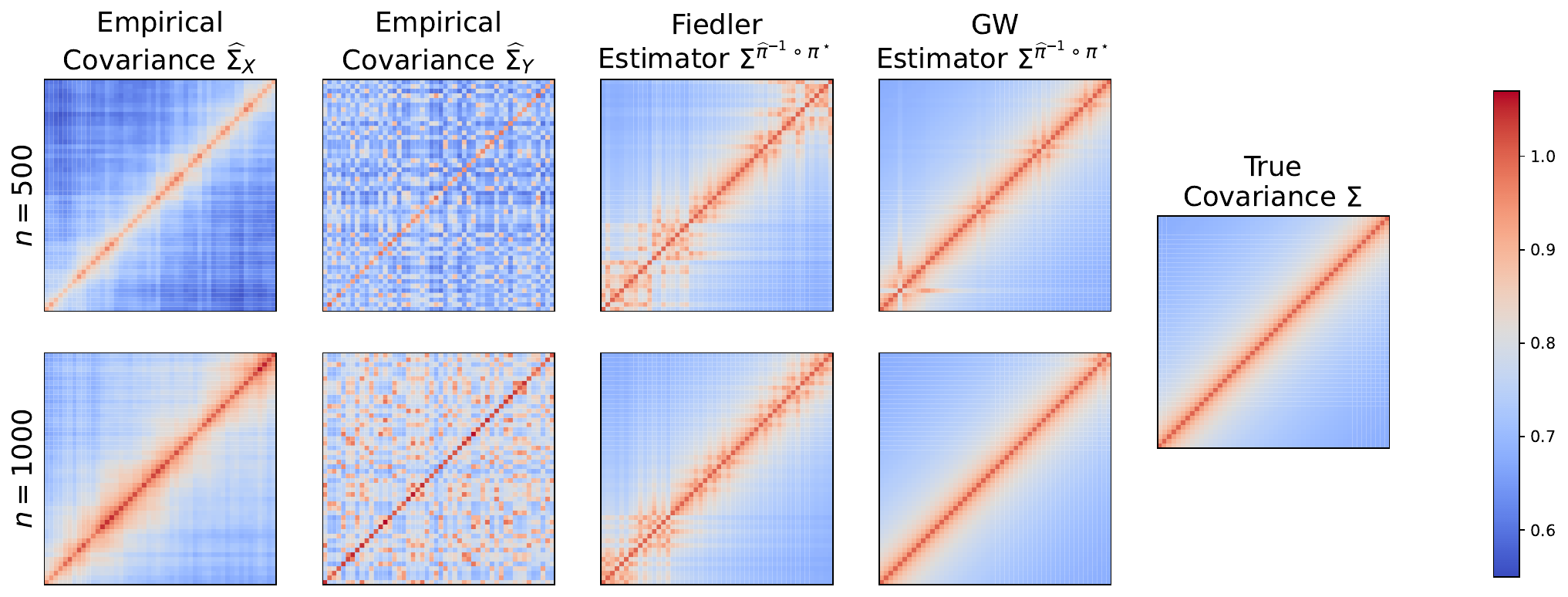}
    \caption{Covariance alignment for structured Robinson covariance matrices $\Sigma \in \mathbb{R}^{d \times d}$ with $\Sigma_{ij} = (1 + |i - j|)^{-\gamma}$ and decaying off-diagonals for dimension $d = 50$. Here we compare the Sinkhorn GW estimator to a classical Fiedler vector spectral estimator. Across all sample sizes $n$ we find that GW outperforms the spectral estimator and almost perfectly reconstructs the structured covariance matrix at $n = 1,000$.}
    \label{fig:structural_example}
\end{figure}

\subsection{Numerical experiments}\label{subsec:numerical}
Although the GW estimator and the QMLE are both minimax optimal, the GW estimator presents significant computational advantages over the QMLE. The objective of the QMLE~\eqref{eq:QMLE} is a general quadratic assignment problem over the set of permutation matrices, and it is NP-hard to find a constant approximation for QAPs in general \cite{sahni1976p}. In contrast, although the GW estimator \eqref{eq:GW_estimator} is also NP-hard, the lifting to probabilistic couplings offers more computational flexibility in the development of numerical heuristics as gradient information can now be used and potentially leads to better convergence results empirically. This approach has been successfully leveraged in a variety of contexts~\cite{SolPeyKim16, peyre2017computational}. As discussed above, the combinatorial optimization \eqref{eq:QMLE} for the QMLE cannot be lifted to probabilistic couplings, for \eqref{eq:QMLE} is a minimization of convex functions and gives a non-permutation after the lifting. In the experiments below, we optimize the GW objective \eqref{eq:GW_estimator} over the space of couplings by adding a small entropic penalty (with penalty parameter $\varepsilon>0$) to the cost and using Sinkhorn projections~\cite{SolPeyKim16, peyre2017computational, flamary2021pot}. At convergence, we round the final coupling to the closest permutation matrix by solving a linear assignment problem; we call it the Sinkhorn GW estimator. We also optimize both GW and QMLE objectives directly over the set of permutation matrices using the commercial LocalSolver optimizer, a state-of-the-art local search mixed-integer program solver~\cite{benoist2010toward}.

We begin with a structured covariance alignment problem where $\Sigma \in \mathbb{R}^{d \times d}$ with $d = 50$ is a Robinson matrix with decaying off-diagonals $\Sigma_{ij} = (1 + |i - j|)^{-\gamma}$ for $\gamma = 0.1$. In the experiment we set $m=n$, and compare two estimators of $\pi^\star$: one is the Sinkhorn GW estimator with $\varepsilon = 5 \times 10^{-4}$, and the other is a classical spectral estimator~\cite{atkins1998spectral} which estimates the permutation $\widehat{\pi}$ by sorting the entries of the Fiedler vector (the eigenvector for the smallest nonzero eigenvalue) of the unnormalized Laplacian matrix $L = \text{diag}(\hS_Y\1_d) - \hS_Y$. We note that the second estimator crucially relies on the Robinson structure of $\Sigma$ where the Fiedler vector is monotone. Figure \ref{fig:structural_example} displays the visualizations for $\hS_X, \hS_Y$, and the permuted covariances for both estimators under two sample sizes $n\in \{500, 1000\}$, and shows that the GW estimator achieves a smaller Frobenius loss than the Fiedler vector estimator. In particular, it is strikingly visible that when $n = 1,000$, the GW estimator almost perfectly recovers the structured covariance $\Sigma$. 

\begin{figure}[t]
    \centering
    \includegraphics[width=\textwidth]{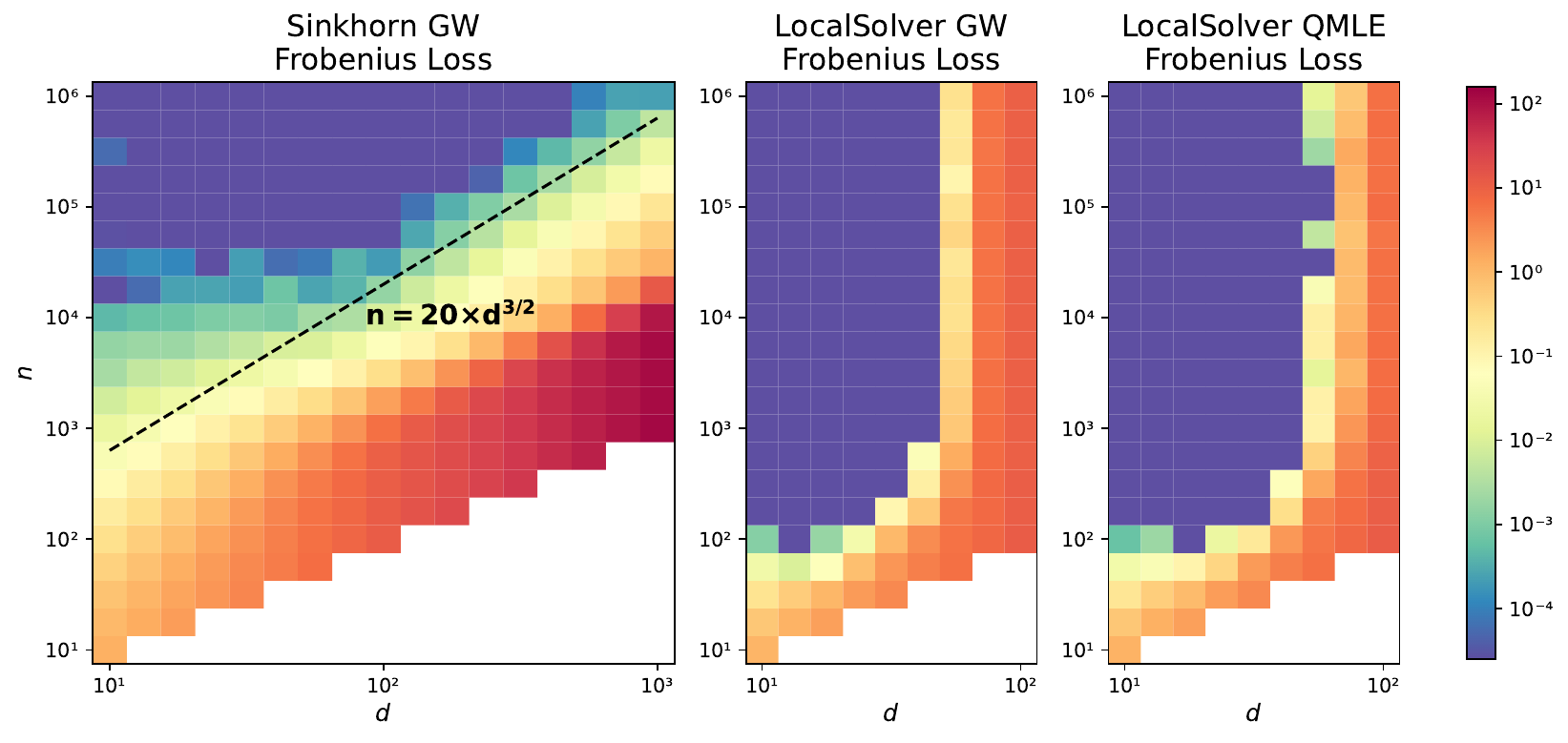}
    \caption{Comparison of Sinkhorn GW, LocalSolver GW, and LocalSolver QMLE for covariance alignment with random Wishart distributed covariances $\Sigma \sim \mathcal{W}_d(I_d, d)$. Here we vary the sample size $m = n \in [10, 10^6]$, and the dimension $d\in [10, 10^3]$ of the covariance matrix. Three heatmaps show the mean squared Frobenius norm of these algorithms as a function of $(d, n)$ averaged over $1,000$ experiments for Sinkhorn GW and $100$ experiments for the more time-consuming LocalSolver estimators. We observe that the LocalSolver combinatorial search often achieves exact recovery for small values of $d$, while only the Sinkhorn GW algorithm scales to high dimensions and attains the minimax optimal scaling $n = \widetilde{\Theta}(d^{3/2})$  for sufficiently large $n$.}
    \label{fig:n_d_scaling}
\end{figure}

We proceed with random covariance matrices $\Sigma$, and compare the Sinkhorn GW estimator with the QMLE and GW estimators computed via LocalSolver. We generate random Wishart distributed covariances $\Sigma \sim \mathcal{W}_d(I_d, d)$ with $d\in [10,10^3]$, choose various sample sizes $m=n\in [10,10^6]$, and set $\varepsilon = 1/d^2$ for the Sinkhorn GW estimator. Figure \ref{fig:n_d_scaling} displays the heatmaps of the mean squared Frobenius losses (averaged over $1,000$ experiments for Sinkhorn GW and $100$ experiments for the more time-consuming LocalSolver estimators) achieved by all estimators for each $(n,d)$ pair. We observe that for small values of $d$, the combinatorial search performed by LocalSolver can often achieve exact recovery (dark purple color) and result in smaller losses than the Sinkhorn GW estimator. However, starting from moderate values of $d$ (say $d=100$), the LocalSolver combinatorial search fails and outputs random results even for very large $n$. In contrast, the Sinkhorn GW estimator still achieves vanishing Frobenius losses for large problems with $d=1,000$, with a running time on the order of several seconds. We also find that for sufficiently large $n$, the minimax optimal scaling $n = \widetilde{\Theta}(d^{3/2})$ (indicated by the black dashed line in Figure \ref{fig:n_d_scaling}) holds for the Sinkhorn GW estimator. 

The above experimental results show that the Sinkhorn GW estimator is competitive for covariance alignment both statistically and computationally. However, we do not have a formal proof that the Sinkhorn GW estimator achieves the minimax rate in Theorems \ref{thm:upper_bound_QMLE}-\ref{thm:lower_bound}. We leave the statistical properties of the Sinkhorn GW estimator, or more generally the construction of a polynomial-time minimax optimal estimator, as outstanding open problems.

\section{Proof of upper bounds}\label{sec:upper_bounds}
In this section we establish the upper bounds for both the QMLE and GW estimators in Theorem \ref{thm:upper_bound_QMLE} and Theorem \ref{thm:upper_bound_GW}, respectively. Our high-level idea is to establish a high-probability lower bound of $R_{m,n}(\widehat{\pi}) - R_{m,n}(\pi^\star)$ for both the QMLE and GW estimators: with probability exceeding $1-\delta$,
\begin{align*}
    R_{m, n}(\widehat{\pi}) - R_{m, n}(\pi^\star) \geq \rho\left(\Fnorm{\Sigma^{\widehat{\pi}} - \Sigma^{\pi^\star}}\right)
\end{align*}
holds for some properly chosen $\rho(\cdot)$. Since $R_{m, n}(\widehat{\pi}) - R_{m, n}(\pi^\star)\le 0$ by definition of $\widehat{\pi}$, solving the above inequality will give the target upper bound of $\Fnorm{\Sigma^{\widehat{\pi}} - \Sigma^{\pi^\star}}$. 

\subsection{Upper bound of the QMLE}\label{sec:QMLE}
We first establish the upper bound for the QMLE and prove Theorem \ref{thm:upper_bound_QMLE}. In Section \ref{subsec:QMLE_good_event}, we define several good events which happen simultaneously with probability at least $1-\delta$. In Section \ref{subsec:QMLE_loss_difference}, we condition on these events and derive a deterministic lower bound of $R_{m, n}^\QMLE(\widehat{\pi}^\QMLE) - R_{m, n}^\QMLE(\pi^\star)$. In Section \ref{subsec:QMLE_L1L2}, we prove a useful trace-Frobenius inequality which relates the above lower bound to the final Frobenius norm.

\subsubsection{Good events}\label{subsec:QMLE_good_event}
Without loss of generality, throughout our proof we assume that $\pi^\star = \id$; by swapping the roles of $X^m$ and $Y^n$ we also assume that $m\ge n$. Denote $\widehat{\pi}^\QMLE$ by $\widehat{\pi}$ for notational simplicity. The good event $\calG$ that we will condition on afterwards is defined to be the intersection of events $\calG = \calE_1 \cap \calE_2 \cap \calE_3$, where
\begin{align}
\calE_1 &:= \bigg\{ \jiao{\hS_Y, (\hS_X^{-1})^{\widehat{\pi}} - \hS_X^{-1}} \ge  \jiao{\Sigma, (\hS_X^{-1})^{\widehat{\pi}} - \hS_X^{-1}} \nonumber \\
& \qquad - c\sqrt{\frac{d\log d+\log(1/\delta)}{n}}\Fnorm{\Sigma^{1/2}((\hS_X^{-1})^{\widehat{\pi}} - \hS_X^{-1})\Sigma^{1/2}} \bigg\}, \label{eq:QMLE_E1} \\
\calE_2 &:= \bigg\{ \jiao{\Sigma^{-1} - \Sigma^{-1}\Sigma^{\widehat{\pi}^{-1}}\Sigma^{-1}, \hS_X} \ge \jiao{\Sigma^{-1} - \Sigma^{-1}\Sigma^{\widehat{\pi}^{-1}}\Sigma^{-1}, \Sigma} \nonumber \\
& \qquad - c\sqrt{\frac{d\log d + \log(1/\delta)}{m}} \Fnorm{\Sigma^{-1/2}\Sigma^{\widehat{\pi}^{-1}}\Sigma^{-1/2} - I_d } \bigg\}, \label{eq:QMLE_E2} \\
\calE_3 &:= \bigg\{ \Opnorm{\Sigma^{-1/2} \hS_X \Sigma^{-1/2} - I} \le \min\bigg\{ c\sqrt{\frac{d+\log(1/\delta)}{m}}, \frac{1}{2} \bigg\}\bigg\}. \label{eq:QMLE_E3}
\end{align}
Here $c>0$ is a numerical constant chosen in the following lemma.

\begin{lemma}\label{lemma:QMLE_good_event}
There exist absolute constants $c,c_0>0$ independent of $(n,m,d,\delta)$ such that, if $m \ge n\ge c_0(d\log d+\log(1/\delta))$, it holds that $\bP(\calG) \ge 1-\delta$. 
\end{lemma}
\begin{proof}
We show that $\bP(\calE_i)\ge 1-\delta/3$ for all $i=1,2,3$. For $\calE_1$, first consider the scenario where $\widehat{\pi}$ is replaced by any fixed $\pi\in S_d$; we denote the resulting event by $\calE_1(\pi)$. After conditioning on $X^m$, the Hanson--Wright inequality in Lemma \ref{lemma:Hanson-Wright} tells that $\mathbb{P}(\calE_1(\pi))\ge 1-\delta/(3\cdot d!)$ for any fixed $\pi$, for a large numerical constant $c>0$. Since $\cap_{\pi\in S_d} \calE_1(\pi)\subseteq \calE_1$, the union bound yields $\mathbb{P}(\calE_1)\ge 1-\delta/3$. The analysis of $\calE_2$ follows from a similar application of the Hanson--Wright inequality to $\hS_X$. 

Finally, the lower bound of $\mathbb{P}(\calE_3)$ follows from the concentration of the sample covariance matrix, cf. \cite[Theorem 5.7]{rigollet2015high}. The upper bound $1/2$ could be ensured by $m\ge c_0(d\log d +\log(1/\delta))$ and a sufficiently large constant $c_0$. 
\end{proof}

The specific form of the good event $\calG$ is motivated by the calculation of $R_{m, n}^\QMLE(\widehat{\pi}) - R_{m, n}^\QMLE(\id)$ in the following section. 

\subsubsection{Lower bound of the loss difference}\label{subsec:QMLE_loss_difference}
Throughout this section we condition on the good event $\calG$ defined in \eqref{eq:QMLE_E1}-\eqref{eq:QMLE_E3} and prove a deterministic lower bound of $R_{m, n}^\QMLE(\widehat{\pi}) - R_{m, n}^\QMLE(\id)$. Specifically, 
\begin{align*}
     &R_{m, n}^\QMLE(\widehat{\pi}) - R_{m, n}^\QMLE(\id) \\
     &\stepa{=} \langle \hS_Y, (\hS_X^{-1})^{\widehat{\pi}} - \hS_X^{-1}\rangle \\
     &\stepb{\ge} \langle\Sigma, (\hS_X^{-1})^{\widehat{\pi}} - \hS_X^{-1}\rangle - c\sqrt{\frac{d\log d + \log(1/\delta)}{n}}\underbrace{ \Fnorm{\Sigma^{1/2}((\hS_X^{-1})^{\widehat{\pi}} - \hS_X^{-1})\Sigma^{1/2}} }_{=: R_1}.
\end{align*}
Here (a) follows from the definition of the empirical loss in \eqref{eq:QMLE}, and (b) is due to the event $\calE_1$ in \eqref{eq:QMLE_E1}. Deferring the analysis of the remainder term $R_1$ for a moment, we continue to lower bound the main term as
\begin{align*}
&\langle\Sigma, (\hS_X^{-1})^{\widehat{\pi}} - \hS_X^{-1}\rangle = \langle \Sigma^{\widehat{\pi}^{-1}} - \Sigma, \hS_X^{-1}\rangle \\
&\stepc{=} \langle \Sigma^{\widehat{\pi}^{-1}} - \Sigma, \Sigma^{-1}\rangle + \jiao{ \Sigma^{\widehat{\pi}^{-1}} - \Sigma, \Sigma^{-1}(\Sigma - \hS_X)\Sigma^{-1} } + \underbrace{\jiao{ \Sigma^{\widehat{\pi}^{-1}} - \Sigma, \hS_X^{-1} - \Sigma^{-1} - \Sigma^{-1}(\Sigma - \hS_X)\Sigma^{-1} } }_{=: R_2}  \\
&\stepd{\ge} \langle \Sigma^{\widehat{\pi}^{-1}} - \Sigma, \Sigma^{-1}\rangle - c\sqrt{\frac{d\log d + \log(1/\delta)}{m}} \Fnorm{\Sigma^{-1/2}\Sigma^{\widehat{\pi}^{-1}}\Sigma^{-1/2} - I_d } + R_2.
\end{align*}
Here (c) is motivated by the approximation $\hS_X^{-1}\approx \Sigma^{-1} + \Sigma^{-1}(\Sigma - \hS_X)\Sigma^{-1}$ if $\hS_X \approx \Sigma$; this step linearizes the dependence on $\hS_X$. As a result, (d) makes use of the event $\calE_2$ in \eqref{eq:QMLE_E2}. 

The remainder terms $R_1$ and $R_2$ are bounded in the following lemma. 

\begin{lemma}\label{lemma:remainder}
Under the event $\calE_3$ in \eqref{eq:QMLE_E3}, it holds that
\begin{align*}
    R_1 &\le \Fnorm{ \Sigma^{1/2}(\Sigma^{-1})^{\widehat{\pi}}\Sigma^{1/2}  - I_d } + 4c\sqrt{d}\cdot \sqrt{\frac{d+\log(1/\delta)}{m}}, \\
    |R_2| &\le 2c^2\cdot \frac{\sqrt{d}(d+\log(1/\delta))}{m}\Fnorm{\Sigma^{-1/2}\Sigma^{\widehat{\pi}^{-1} }\Sigma^{-1/2} - I_d}. 
\end{align*}
\end{lemma}

Plugging Lemma \ref{lemma:remainder} into the previous lower bound yields
\begin{equation}\label{eq:QMLE_loss_difference}
\begin{aligned}
    & R_{m, n}^\QMLE(\widehat{\pi}) - R_{m, n}^\QMLE(\id) \\
    &\ge \langle \Sigma^{\widehat{\pi}^{-1}} - \Sigma, \Sigma^{-1}\rangle - C\sqrt{\frac{d(d\log d + \log(1/\delta))(d + \log(1/\delta))}{mn}} \\
    &\quad - C\sqrt{\frac{d\log d + \log(1/\delta)}{n}}\Fnorm{ \Sigma^{1/2}(\Sigma^{-1})^{\widehat{\pi}}\Sigma^{1/2}  - I_d } \\
    &\quad - C\left(\sqrt{\frac{d\log d + \log(1/\delta)}{m}} + \frac{\sqrt{d}(d+\log(1/\delta))}{m}\right)\Fnorm{\Sigma^{-1/2}\Sigma^{\widehat{\pi}^{-1} }\Sigma^{-1/2} - I_d},
\end{aligned}
\end{equation}
for some absolute constant $C>0$ independent of $(n,m,d)$.

\subsubsection{A trace-Frobenius inequality}\label{subsec:QMLE_L1L2}

Based on \eqref{eq:QMLE_loss_difference}, the remaining question is to relate the Frobenius norms in \eqref{eq:QMLE_loss_difference} to the inner product $\langle\Sigma^{\widehat{\pi}^{-1}} - \Sigma, \Sigma^{-1}\rangle$. Roughly speaking, for the matrix $A = \Sigma^{-1/2}\Sigma^{\widehat{\pi}^{-1} }\Sigma^{-1/2} - I_d$, the Frobenius norm $\Fnorm{A}$ corresponds to the $L_2$ norm of its eigenvalues, and the inner product $\langle\Sigma^{\widehat{\pi}^{-1}} - \Sigma, \Sigma^{-1}\rangle = \trace(A)$ corresponds to the sum of its eigenvalues. Although $\Fnorm{A}\le \trace(A)$ holds for PSD matrices $A\succeq 0$, this inequality may break down if some eigenvalues of $A$ are negative. 

To this end, we present a useful technical inequality in the following lemma. 
\begin{lemma}\label{lemma:L1_L2}
Let $x_1,\cdots,x_d > 0$ with $\prod_{i=1}^d x_i = 1$. Then
\begin{align*}
\sum_{i=1}^d (x_i-1)^2 \le 4\left(\sum_{i=1}^d (x_i-1) + \left(\sum_{i=1}^d (x_i-1)\right)^2 \right). 
\end{align*}
\end{lemma}

The usefulness of Lemma \ref{lemma:L1_L2} lies in the following application: let $x_1,\cdots,x_d$ be all eigenvalues of $\Sigma^{-1/2}\Sigma^{\widehat{\pi}^{-1}}\Sigma^{-1/2}$, then $x_1,\cdots,x_d>0$, and $\prod_{i=1}^d x_i = \det(\Sigma^{\widehat{\pi}^{-1}})\det(\Sigma^{-1})=1$. Therefore, we arrive at the following corollary, which we term as a \emph{trace-Frobenius inequality}. 

\begin{corollary}\label{cor:trace_frob}
For positive definite $\Sigma$, the following inequalities hold: 
\begin{align*}
    \Fnorm{\Sigma^{-1/2}\Sigma^{\widehat{\pi}^{-1} }\Sigma^{-1/2} - I_d}^2 &\le 4(\langle\Sigma^{\widehat{\pi}^{-1}} - \Sigma, \Sigma^{-1}\rangle + \langle\Sigma^{\widehat{\pi}^{-1}} - \Sigma, \Sigma^{-1}\rangle^2), \\
    \Fnorm{\Sigma^{1/2}(\Sigma^{-1})^{\widehat{\pi}}\Sigma^{1/2} - I_d}^2 &\le 4(\langle\Sigma^{\widehat{\pi}^{-1}} - \Sigma, \Sigma^{-1}\rangle + \langle\Sigma^{\widehat{\pi}^{-1}} - \Sigma, \Sigma^{-1}\rangle^2). 
\end{align*}
\end{corollary}

By \eqref{eq:QMLE_loss_difference} and Corollary \ref{cor:trace_frob}, we conclude that
\begin{align}\label{eq:QMLE_difference_final}
&R_{m, n}^\QMLE(\widehat{\pi}) - R_{m, n}^\QMLE(\id) \ge (1-u)\langle\Sigma^{\widehat{\pi}^{-1}} - \Sigma, \Sigma^{-1}\rangle  - u\sqrt{\langle\Sigma^{\widehat{\pi}^{-1}} - \Sigma, \Sigma^{-1}\rangle } - v
\end{align}
where
\begin{align*}
u &= 2C\left(\sqrt{\frac{d\log d + \log(1/\delta)}{n}} + \sqrt{\frac{d\log d + \log(1/\delta)}{m}} + \frac{\sqrt{d}(d+\log(1/\delta))}{m}\right), \\
v & =C\sqrt{\frac{d(d\log d + \log(1/\delta))(d + \log(1/\delta))}{mn}}. 
\end{align*}
By our assumption $m\ge n\ge c_0(d\log d+\log(1/\delta))$ and $mn\ge c_0d(d\log d + \log(1/\delta))(d + \log(1/\delta))$, choosing a large enough numerical constant $c_0>0$ ensures that $u,v\le 1/2$. Moreover, by definition of $\widehat{\pi} = \widehat{\pi}^\QMLE$ we have $R_{m, n}^\QMLE(\widehat{\pi}) - R_{m, n}^\QMLE(\id) \le 0$. Hence, by \eqref{eq:QMLE_difference_final} and Lemma \ref{lemma:quad_bound}, we conclude that
\begin{align}\label{eq:QMLE_trace_upperbound}
\jiao{\Sigma^{\widehat{\pi}^{-1}} - \Sigma, \Sigma^{-1}} &\le \left( \frac{u}{1-u} + \sqrt{\frac{v}{1-u}}\right)^2 \le 4(u+\sqrt{v})^2 \le 8(u^2 + v).
\end{align}

From \eqref{eq:QMLE_trace_upperbound} we are ready to prove Theorem \ref{thm:upper_bound_QMLE}: conditioned on $\calG$, 
\begin{align*}
\Fnorm{\Sigma^{\widehat{\pi}} - \Sigma}^2 & = \Fnorm{\Sigma^{\widehat{\pi}^{-1}} - \Sigma}^2 \stepa{\le} \Fnorm{\Sigma^{-1/2}\Sigma^{\widehat{\pi}^{-1} }\Sigma^{-1/2} - I_d}^2 \\
&\stepb{\le} 4(\langle\Sigma^{\widehat{\pi}^{-1}} - \Sigma, \Sigma^{-1}\rangle + \langle\Sigma^{\widehat{\pi}^{-1}} - \Sigma, \Sigma^{-1}\rangle^2) \\
&\stepc{\le} 288(u^2+v) \\
&\stepd{=} O\left(\frac{d\log d + \log(1/\delta)}{n} + \sqrt{\frac{d(d\log d + \log(1/\delta))(d + \log(1/\delta))}{mn}} \right), 
\end{align*}
where (a) is due to the assumption $\Opnorm{\Sigma}\le 1$, (b) uses Corollary \ref{cor:trace_frob}, (c) follows from \eqref{eq:QMLE_trace_upperbound} and that $u^2+v<1$, and (d) plugs in the definition of $(u,v)$ and uses $m\ge \sqrt{mn}\ge \sqrt{c_0d}(d+\log(1/\delta))$. Since $\mathbb{P}(\calG)\ge 1-\delta$ by Lemma \ref{lemma:QMLE_good_event}, this proves the upper bound of Theorem \ref{thm:upper_bound_QMLE}.

\subsection{Upper bound of the GW estimator}
Now we establish the upper bound for the GW estimator and prove Theorem~\ref{thm:upper_bound_GW} using similar arguments as in our proof for the QMLE. Again, in Section~\ref{subsec:GW_good_event} we define the good events, and in Section~\ref{subsec:GW_loss_difference} we condition on the good events and derive a deterministic lower bound of the loss difference $R_{m, n}^\GW(\widehat{\pi}^\GW) - R_{m, n}^\GW(\pi^\star)$. 

\subsubsection{Good events}\label{subsec:GW_good_event}
As before, without loss of generality we assume that $\pi^* = \id$ and refer to $\widehat{\pi}^\GW$ as $\widehat{\pi}$ for notational convenience. We condition on a good event $\cG = \calE_1 \cap \calE_2 \cap \calE_3$ defined as the following intersection: 
\begin{align}
\calE_1 &:= \bigg\{ \jiao{\hS_Y, \hS_X - (\hS_X)^{\widehat{\pi}}} \ge  \jiao{\Sigma, \hS_X - (\hS_X)^{\widehat{\pi}}} \nonumber \\
& \qquad - c\sqrt{\frac{d\log d+\log(1/\delta)}{n}}\Fnorm{\Sigma^{1/2}(\hS_X - (\hS_X)^{\widehat{\pi}})\Sigma^{1/2}} \bigg\}, \label{eq:GW_E1} \\
\calE_2 &:= \bigg\{ \jiao{\Sigma, \hS_X - (\hS_X)^{\widehat{\pi}}} \ge  \jiao{\Sigma, \Sigma - \Sigma^{\widehat{\pi}}} \nonumber \\
& \qquad - c\sqrt{\frac{d\log d+\log(1/\delta)}{m}}\Fnorm{\Sigma^{1/2}(\Sigma - \Sigma^{\widehat{\pi}^{-1} })\Sigma^{1/2}} \bigg\}, \label{eq:GW_E2} \\
\calE_3 &:= \bigg\{ \Fnorm{\hS_X - (\hS_X)^{\widehat{\pi}}} \le \Fnorm{\Sigma - \Sigma^{\widehat{\pi}}} + c\sqrt{\frac{d(d+\log(1/\delta))}{m}}\bigg\}. \label{eq:GW_E3}
\end{align}
Here $c>0$ is a numerical constant chosen in the following lemma.

\begin{lemma}\label{lemma:GW_good_event}
Let $m \wedge n\ge d\log d+\log(1/\delta)$. There exists absolute constant $c > 0$ independent of $(n,m,d,\delta)$ such that $\bP(\calG) \ge 1-\delta$. 
\end{lemma}
\begin{proof}
The high probabilities of events $\calE_1, \calE_2$ follows from similar applications of the Hanson-Wright inequality as in the proof of Lemma \ref{lemma:QMLE_good_event}. For $\calE_3$, applying the triangle inequality twice yields
    \begin{align*}
        \Fnorm{\hS_X - (\hS_X)^{\widehat{\pi}}} - \Fnorm{\Sigma - \Sigma^{\widehat{\pi}}} \leq \Fnorm{(\hS_X - \Sigma)^{\widehat{\pi}} - (\hS_X - \Sigma)} \leq 2\Fnorm{\hS_X - \Sigma},
    \end{align*}
    where in the last step we have used the permutation invariance of the Frobenius norm. Since
    \begin{align*}
    \Fnorm{\hS_X - \Sigma} \le \Opnorm{\Sigma}\Fnorm{\Sigma^{-1/2}\hS_X\Sigma^{-1/2} - I_d}\le \sqrt{d}\Opnorm{\Sigma^{-1/2}\hS_X\Sigma^{-1/2} - I_d}, 
    \end{align*}
    the rest of the proof follows from \eqref{eq:QMLE_E3} and Lemma \ref{lemma:QMLE_good_event}. 
\end{proof}

\subsubsection{Lower bound of the loss difference}\label{subsec:GW_loss_difference}
Again, we condition on the good event $\calG$ defined in \eqref{eq:GW_E1}-\eqref{eq:GW_E3} and prove a deterministic lower bound of $R_{m, n}^\GW(\widehat{\pi}) - R_{m, n}^\GW(\id)$. We begin by writing
\begin{align*}
    R_{m, n}^\GW(\widehat{\pi}) - R_{m, n}^\GW(\id) &\stepa{=} 2\jiao{\hS_Y, \hS_X - (\hS_X)^{\widehat{\pi}}}\\
    &\stepb{\ge} 2\jiao{\Sigma, \hS_X - (\hS_X)^{\widehat{\pi}}} - 2c\sqrt{\frac{d\log d+\log(1/\delta)}{n}}\underbrace{ \Fnorm{\Sigma^{1/2}(\hS_X - (\hS_X)^{\widehat{\pi}})\Sigma^{1/2}} }_{=: R_1} \\
    &\stepc{\ge} 2\jiao{\Sigma, \Sigma - \Sigma^{\widehat{\pi}}} - 2c\sqrt{\frac{d\log d+\log(1/\delta)}{n}}\left(R_1+\underbrace{ \Fnorm{\Sigma^{1/2}(\Sigma - \Sigma^{\widehat{\pi}^{-1} })\Sigma^{1/2}}}_{=:R_2 }\right) \\
    &\stepd{=} \Fnorm{\Sigma^{\widehat{\pi}} - \Sigma}^2 - 2c\sqrt{\frac{d\log d+\log(1/\delta)}{n}}(R_1 + R_2). 
\end{align*}
Here (a) follows from the definition of the empirical loss in~\eqref{eq:GW_estimator}, (b) is due to the event $\calE_1$ in~\eqref{eq:GW_E1}, (c) is due to the event $\calE_2$ in \eqref{eq:GW_E2} and $m\ge n$, and (d) uses the identity
\begin{align*}
    2\jiao{\Sigma, \Sigma - \Sigma^{\widehat{\pi}}} = \Fnorm{\Sigma}^2 + \Fnorm{\Sigma^{\widehat{\pi}}}^2 - 2\jiao{\Sigma, \Sigma^{\widehat{\pi}}} = \Fnorm{\Sigma^{\widehat{\pi}} - \Sigma}^2. 
\end{align*}

Next we upper bound the remainder terms. 
\begin{itemize}
    \item The remainder term $R_1$: by the event $\calE_3$ in \eqref{eq:GW_E3} and $\Opnorm{\Sigma}\le 1$, it directly follows that
    \begin{align*}
        R_1 \le \Opnorm{\Sigma}\Fnorm{\hS_X - (\hS_X)^{\widehat{\pi}}} \le \Fnorm{\Sigma - \Sigma^{\widehat{\pi}}} + c\sqrt{\frac{d(d+\log(1/\delta))}{m}}. 
    \end{align*}
    \item The remainder term $R_2$: by the permutation invariance of the Frobenius norm, and by $\Opnorm{\Sigma}\le 1$ again, we have
    \begin{align*}
        R_2 \le \Opnorm{\Sigma}\Fnorm{\Sigma - \Sigma^{\widehat{\pi}^{-1} }} \le \Fnorm{\Sigma - \Sigma^{\widehat{\pi}^{-1} }} = \Fnorm{\Sigma - \Sigma^{\widehat{\pi} }}. 
    \end{align*}
\end{itemize}

A combination of the above displays tells that
\begin{align}\label{eq:GW_loss_difference}
    R_{m, n}^\GW(\widehat{\pi}) - R_{m, n}^\GW(\id) &\ge \Fnorm{\Sigma^{\widehat{\pi}} - \Sigma}^2 - u\Fnorm{\Sigma^{\widehat{\pi}} - \Sigma} - v, 
\end{align}
where
\begin{align*}
    u = C\sqrt{\frac{d\log d+\log(1/\delta)}{n}}, \qquad v = C\sqrt{\frac{d(d\log d + \log(1/\delta))(d + \log(1/\delta)))}{mn}}, 
\end{align*}
and $C>0$ is an absolute constant independent of $(n,m,d)$.

Finally, by definition of $\widehat{\pi} = \widehat{\pi}^\GW$ we have $R_{m, n}^\GW(\widehat{\pi}) - R_{m, n}^\GW(\id) \le 0$. Hence, by \eqref{eq:GW_loss_difference} and Lemma \ref{lemma:quad_bound} we conclude that when conditioned on $\calG$, 
\begin{align*}
\Fnorm{\Sigma^{\widehat{\pi}} - \Sigma}^2 &\le (u + \sqrt{v})^2 \le 2(u^2 + v) \\
&= O\left(\frac{d\log d + \log(1/\delta)}{n} + \sqrt{\frac{d(d\log d + \log(1/\delta))(d + \log(1/\delta))}{mn}} \right).
\end{align*}
The above inequality together with $\mathbb{P}(\calG)\ge 1-\delta$ in Lemma \ref{lemma:GW_good_event} completes the proof of Theorem \ref{thm:upper_bound_GW}. 

\section{Proof of lower bounds}\label{sec:lower_bound}
In this section we prove the minimax lower bound in Theorem \ref{thm:lower_bound}, and by symmetry we assume that $m\ge n$. We use an information-theoretic argument: after specifying the prior distribution on $(\Sigma,\pi^\star)$ in Section \ref{subsec:prior}, we invoke Fano's inequality and upper bound the mutual information $I(\pi^\star; X^m, Y^n)$ in Section \ref{subsec:mutual_info}. The proof of the mutual information upper bound consists of three steps: 
\begin{enumerate}
    \item By the variational representation of the mutual information, we first reduce the weak recovery of $\pi^\star$ to the detection of $\Sigma$. This step is similar to \cite{hall2023partial,wu2022settling};
    \item By the superadditivity of the KL divergence between a stationary process and an iid process, we reduce the mutual information to the case with an equal sample size, where we also have $m$ observations from $Y$; 
    \item By passing to the $\chi^2$-divergence, we carry out the second moment computation of the likelihood ratio and derive the final upper bound of $I(\pi^\star; X^m, Y^n)$. 
\end{enumerate}

\subsection{Construction of the prior}\label{subsec:prior}
To prove a Bayes risk lower bound, we impose the following prior distribution on $(\Sigma,\pi^\star)$. The choice of the prior on $\pi^\star$ is natural: we simply take $\pi^\star\sim \text{Unif}(S_d)$ to be a uniformly random permutation. As for the prior distribution over $\Sigma$, we set
\begin{align*}
\Sigma = \frac{1}{2}(I_d + \eta S), 
\end{align*}
with some parameter
\begin{align}\label{eq:eta_range}
    \eta \in \left(0, \frac{1}{2c_1\sqrt{d}}\right)
\end{align}
to be chosen later (here $c_1>0$ is an absolute constant appearing in Lemma \ref{lemma:S0_construction}), and a random matrix $S\sim \text{Unif}(\calS_0)$. The set of matrices $\calS_0$ is constructed so that the following three properties hold: 
\begin{enumerate}
    \item It is a subset of symmetric Rademacher matrices, i.e. $\calS_0 \subseteq \calS \triangleq \{S \in \{\pm 1\}^{d\times d}: S = S^\top \}$; 
    \item Every matrix in $\calS_0$ has a small operator norm: $\Opnorm{S}\le c_1\sqrt{d}$ for all $S\in \calS_0$; 
    \item Different permutations to a matrix in $\calS_0$ give different matrices: for all $S\in \calS_0$ and $\pi_1, \pi_2\in S_d$ with $\sum_{i=1}^d \mathbbm{1}(\pi_1(i)\neq \pi_2(i)) \ge d/10$, it holds that $\Fnorm{S^{\pi_1} - S^{\pi_2}}^2 \ge c_2d^2.$
\end{enumerate}
Note that the choice of $\eta$ in \eqref{eq:eta_range} and the second property of $\calS_0$ ensure that $\Opnorm{\Sigma}\le 1$ almost surely. The next lemma shows that $\calS_0$ can be constructed as a sufficiently large subset of $\calS$. 

\begin{lemma}\label{lemma:S0_construction}
There exist absolute constants $c_1,c_2,d_0>0$ such that the following holds: if $d\ge d_0$, there exists some $\calS_0\subseteq \calS$ such that all the above properties hold, and
\begin{align*}
    \frac{|\calS_0|}{|\calS|} \ge \frac{1}{2}. 
\end{align*}
\end{lemma}

We proceed to lower bound the Bayes risk under the above prior, and to this end we introduce the (generalized) Fano's inequality. Consider a general scenario where $\theta\in\Theta$ is an unknown parameter, the learner observes $X\sim P_\theta$, and $L: \Theta\times \calA\to \bR_+$ is a non-negative loss function with a generic action space $\mathcal{A}$. The following lemma presents a general lower bound of the Bayes risk. 

\begin{lemma}\label{lemma:Fano}
Let $\pi$ be any probability distribution over $\Theta$. For any $\Delta>0$, define
\begin{align*}
p_{\Delta} := \sup_{a} \pi\{ \theta\in \Theta: L(\theta,a)\le \Delta \}. 
\end{align*}
Then for $\theta\sim \pi$, the following Bayes risk lower bound holds: 
\begin{align*}
\inf_{T(\cdot)}\bE_{\pi} \bE_{\theta}[L(\theta, T(X))] \ge \Delta\left(1 - \frac{I(\theta;X) + \log 2}{\log(1/p_\Delta)}\right).
\end{align*}
\end{lemma}

Lemma \ref{lemma:Fano} was essentially proved in \cite{duchi2013distance,chen2016bayes}, and for completeness we include the proof of the above form in Appendix~\ref{subsec:Fano_proof}. To apply this lemma, we will choose $\theta = \pi^\star$, so that $P_{\pi^\star}$ is the mixture distribution $\bE_{\Sigma}[\cN(0,\Sigma)^{\otimes m}\otimes \cN(0,\Sigma^{\pi^\star})^{\otimes n}]$. The loss function $L(\pi^\star,\widehat{\pi})$ is defined as
\begin{align}\label{eq:loss_function}
L(\pi^\star,\widehat{\pi}) = \min_{\Sigma = (I_d+\eta S)/2: S\in \calS_0} \Fnorm{\Sigma^{\pi^\star} - \Sigma^{\widehat{\pi}}}^2 = \frac{\eta^2}{4}\min_{S\in \calS_0} \Fnorm{S^{\pi^\star} - S^{\widehat{\pi}}}^2. 
\end{align}
The next lemma provides an upper bound on the quantity $p_\Delta$ used in Lemma \ref{lemma:Fano}. 

\begin{lemma}\label{lemma:p_Delta}
Let the loss $L$ be given in \eqref{eq:loss_function}, and $p_\Delta$ be defined in Lemma \ref{lemma:Fano}. Then for $\Delta = c_2\eta^2d^2/4$, it holds that $p_\Delta \le d^{-c_3d}$, where $c_3>0$ is an absolute numerical constant. 
\end{lemma}

By Lemma \ref{lemma:Fano} and Lemma \ref{lemma:p_Delta}, as well as the definition of $L$ in \eqref{eq:loss_function}, we have
\begin{align}\label{eq:Fano_lower_bound}
\bE\left[ \Fnorm{\Sigma^{\pi^\star} - \Sigma^{\widehat{\pi}}}^2 \right] \ge \bE\left[ L(\pi^\star, \widehat{\pi}) \right] \ge \frac{c_2\eta^2 d^2}{4}\left(1 - \frac{I(\pi^\star; X^m, Y^n) +\log 2}{c_3d\log d}\right). 
\end{align}
It remains to upper bound the mutual information $I(\pi^\star; X^m, Y^n)$, which is deferred to Section \ref{subsec:mutual_info}.

\subsection{Upper bounding the mutual information}\label{subsec:mutual_info}
In this section we aim to prove the following upper bound of mutual information: 
\begin{align}\label{eq:mutual_info_target}
I(\pi^\star; X^m, Y^n) \le c_4\min\left\{nd^2\eta^2, mnd^2\eta^4 + 1\right\},
\end{align}
where $c_4>0$ is an absolute constant independent of $(n,m,d,\eta)$. We first show that \eqref{eq:mutual_info_target} implies the desired minimax lower bound in Theorem \ref{thm:lower_bound}. To see it, we choose
\begin{align*}
    \eta = c_5\max\left\{\sqrt{\frac{\log d}{nd}}, \left(\frac{\log d}{mnd}\right)^{1/4}   \right\}
\end{align*}
with a small absolute constant $c_5>0$, then Theorem \ref{thm:lower_bound} is a direct consequence of \eqref{eq:Fano_lower_bound}. Also note that since $m\ge d$ and $n\ge \log d$, our choice of $\eta$ satisfies the constraint in \eqref{eq:eta_range}. 

In the rest of this section we prove \eqref{eq:mutual_info_target}. The $O(nd^2\eta^2)$ upper bound essentially corresponds to the known covariance case where $m=\infty$, and is easily established in the following lemma. 

\begin{lemma}\label{lemma:simple_upperbound}
Regardless of the value of $m$, it always holds that
\begin{align*}
I(\pi^\star; X^m, Y^n) \le \frac{nd^2\eta^2}{2}. 
\end{align*}
\end{lemma}

In the sequel we establish the $O(mnd^2\eta^4 + 1)$ upper bound.

\subsubsection{Reduction to the case $m=n$}\label{subsec:reduction}
We first prove an upper bound of $I(\pi^\star; X^m, Y^n)$ which no longer involves $\pi^\star$; the high-level idea is to relate the recovery of $\pi^\star$ (under the loss $L(\pi^\star,\widehat{\pi}$)) to a detection problem of telling if $X^m, Y^n$ are i.i.d. drawn from $ \calN(0,I_d/2)$. To this end, we write
\begin{align*}
I(\pi^\star; X^m, Y^n) &= I(\pi^\star; X^m) + I(\pi^\star; Y^n\mid X^m) \\
&\stepa{=} I(\pi^\star; Y^n\mid X^m) \\
&\stepb{=} \bE_{X^m, \pi^\star}[D_{\text{KL}}( P_{Y^n\mid \pi^\star, X^m} \| \calN(0,I_d/2)^{\otimes n} ) ]  - \bE_{X^m}[D_{\text{KL}}(P_{Y^n\mid X^m} \| \calN(0,I_d/2)^{\otimes n})] \\
&\le \bE_{X^m, \pi^\star}[D_{\text{KL}}( P_{Y^n\mid \pi^\star, X^m} \| \calN(0,I_d/2)^{\otimes n} ) ] \\
&\stepc{=}  \bE_{X^m, \pi^\star}[D_{\text{KL}}( \Phi^{\pi^\star}_\# P_{Y^n\mid X^m} \| \Phi^{\pi^\star}_\# \calN(0,I_d/2)^{\otimes n} ) ]\\
&\stepd{=} \bE_{X^m}[D_{\text{KL}}( P_{Z^n\mid X^m} \| \calN(0,I_d/2)^{\otimes n} ) ]. 
\end{align*}
Here (a) is due to the independence between $\pi^\star$ and $X^m$, (b) follows from simple algebra (this step is also known as the variational representation of the mutual information). In (c), we define a bijection $\Phi^{\pi^\star}: \mathbb{R}^d\to \mathbb{R}^d$ with
\begin{align*}
    \Phi^{\pi^\star}(y_1,\cdots,y_d) = (y_{(\pi^\star)^{-1}(1)}, \cdots, y_{(\pi^\star)^{-1}(d)}). 
\end{align*}
This definition is also extended in the natural way to $(\mathbb{R}^d)^n$ as $\Phi^{\pi^\star}(Y_1,\cdots,Y_n) = (\Phi^{\pi^\star}Y_1,\cdots,\Phi^{\pi^\star}Y_n)$. Clearly $\Phi^{\pi^\star}$ is a bijection, so (c) holds. As for (d), we introduce an auxiliary sequence of random vectors $Z_1, Z_2, \cdots$, which conditioned on $\Sigma$ are i.i.d. drawn from $\calN(0,\Sigma)$. Since $\pi^\star$ is independent of $X^m$, it is clear that $\Phi^{\pi^\star}_\# P_{Y^n\mid X^m} = P_{Z^n\mid X^m}$. In addition, $\Phi^{\pi^\star}_\# \calN(0,I_d/2)^{\otimes n} = \calN(0,I_d/2)^{\otimes n}$ trivially holds, so we arrive at (d). Note that the final expression no longer depends on $\pi^\star$. 

In the second step, we note that $P_{Z^n\mid X^m} = \bE_{\Sigma\mid X^m}[P_{Z^n\mid \Sigma}]$ is a mixture of product distributions and therefore exchangeable. Meanwhile, the second argument $\calN(0, I_d/2)^{\otimes n}$ in the KL divergence is a product distribution. This structure enables us to prove the following lemma and reduce two parameters $(m,n)$ to a single parameter $m$. 

\begin{lemma}\label{lemma:superadditivity}
Let $(Z_1,Z_2,\cdots)$ be a stationary process and $Q$ be an arbitrary distribution. Then the sequence $a_n = D_{\text{\rm KL}}(P_{Z^n}\| Q^{\otimes n})$ satisfies that $\{a_{n+1}-a_n\}$ is non-decreasing, i.e.
\begin{align*}
    a_1 \le a_2 - a_1 \le a_3-a_2\le \cdots.
\end{align*}
In particular, if $m\ge n$, then $a_m/m \ge a_n/n$. 
\end{lemma}

Consequently, we proceed the upper bound as
\begin{align*}
I(\pi^\star; X^m, Y^n) &\le \bE_{X^m}[D_{\text{KL}}( P_{Z^n\mid X^m} \| \calN(0,I_d/2)^{\otimes n} ) ] \\
&\stepe{\le} \frac{n}{m}\bE_{X^m}[D_{\text{KL}}( P_{Z^m\mid X^m} \| \calN(0,I_d/2)^{\otimes m} ) ] \\
&\stepf{=} \frac{n}{m}\left( D_{\text{KL}}(P_{X^m,Z^m} \| \calN(0,I_d/2)^{\otimes (2m)}) - D_{\text{KL}}(P_{X^m} \| \calN(0,I_d/2)^{\otimes m})  \right) \\
&\le \frac{n}{m}D_{\text{KL}}(P_{X^m,Z^m} \| \calN(0,I_d/2)^{\otimes (2m)}), 
\end{align*}
where (e) follows from Lemma \ref{lemma:superadditivity} and $m\ge n$, and (f) is the chain rule of KL divergence. In summary, we have shown that
\begin{align}\label{eq:KL_equal_sample}
I(\pi^\star; X^m, Y^n) \le \frac{n}{m}D_{\text{KL}}(\bE_{\Sigma}[\calN(0,\Sigma)^{\otimes (2m)}] \| \calN(0,I_d/2)^{\otimes (2m)}). 
\end{align}
Note that in \eqref{eq:KL_equal_sample} we essentially reduce to the case with equal sample sizes. Without this reduction we would only have $I(\pi^\star; X^m, Y^n)\le D_{\text{KL}}(\bE_{\Sigma}[\calN(0,\Sigma)^{\otimes (n+m)}] \| \calN(0,I_d/2)^{\otimes (n+m)})$, which could be shown to be loose compared with the target upper bound in \eqref{eq:mutual_info_target}. 

\subsubsection{Second moment method}\label{subsec:second_moment_method}
Based on \eqref{eq:KL_equal_sample} we further have
\begin{align*}
I(\pi^\star; X^m, Y^n) &\stepa{\le} \frac{n}{m}\log(1+\chi^2(\bE_{\Sigma}[\calN(0,\Sigma)^{\otimes (2m)}] \| \calN(0,I_d/2)^{\otimes (2m)}))
\end{align*}
which follows from $D_{\text{KL}}(P\|Q)\le \log(1+\chi^2(P\|Q))$. Now to bound the $\chi^2$-divergence we derive the following lemma based on the standard second moment computation (cf. \cite[Appendix A]{fan2015estimation}) which we prove for completeness in Appendix~\ref{subsec:second_moment_proof}.
\begin{lemma}\label{lemma:second_moment}
Let $\Sigma_S$ be a symmetric matrix in $\mathbb{R}^{d \times d}$ indexed by a random variable $S\sim \mu$. Further assume that $0\prec \Sigma_S \prec 2I_d$ almost surely. Then for $p \in \mathbb{N}$, we have the identity
\begin{align*}
    \chi^2(\bE_{S}[\calN(0,\Sigma_S)^{\otimes p}] \| \calN(0, I_d)^{\otimes p}) + 1 = \bE_{S, T}\Big[\det\Big(I - (\Sigma_S - I)(\Sigma_T - I)\Big)^{-\frac{p}{2}}\Big], 
\end{align*}
where $S, T \sim \mu$ are independent. 
\end{lemma}

Using Lemma \ref{lemma:second_moment} we can write
\begin{align*}
\log(1+\chi^2(\bE_{\Sigma}[\calN(0,\Sigma)^{\otimes (2m)}] \| \calN(0,I_d/2)^{\otimes (2m)})) &\stepb{=} \log \bE_{S,T}[\det(I_d - \eta^2 ST)^{-m}], 
\end{align*}
with $T$ being an independent copy of $S$. Here (b) applies Lemma \ref{lemma:second_moment} to $\Sigma_S = I_d + \eta S$, where we have also used the invariance of the $\chi^2$-divergence with respect to bijections. The condition of Lemma \ref{lemma:second_moment} is fulfilled thanks to $\eta < 1/(2c_1\sqrt{d})$ from \eqref{eq:eta_range} and $\Opnorm{S}\le c_1\sqrt{d}$ from the second property of $\calS_0$. 

To proceed, let $\lambda_1,\cdots,\lambda_d$ be the (positive real) eigenvalues of $ST$, and we write
\begin{align*}
\bE_{S,T}[\det(I_d - \eta^2 ST)^{-m}] &= \bE_{S,T}\left[\exp\left(-m\sum_{i=1}^d \log(1-\eta^2 \lambda_i)\right) \right] \\
&\stepc{\le} \bE_{S,T}\left[\exp\left(-m\sum_{i=1}^d (-\eta^2 \lambda_i - \eta^4\lambda_i^2)\right) \right] \\
&= \bE_{S,T}\left[\exp\left(m\eta^2 \jiao{S,T} + m\eta^4 \Fnorm{ST}^2 \right) \right] \\
&\stepd{\le} \exp(c_1^2 md^3\eta^4)\cdot \bE_{S,T}\left[\exp\left(m\eta^2 \jiao{S,T}\right)\right]
\end{align*}
where in (c) we have used that
\begin{align*}
    \eta^2|\lambda_i| \le \left(\frac{1}{2c_1\sqrt{d}}\right)^2 \Opnorm{S}\Opnorm{T} \le \frac{1}{4}
\end{align*}
from \eqref{eq:eta_range} and the second property for $\calS_0$, as well as $\log(1+x)\ge x - x^2$ for all $|x|\le 1/4$; (d) follows from $\Fnorm{ST}\le \Opnorm{S}\Fnorm{T} \le c_1d^{3/2}$ almost surely. For the remaining quantity, we have
\begin{align*}
\bE_{S,T}\left[\exp\left(m\eta^2 \jiao{S,T}\right)\right] &\le \left(\frac{|\calS|}{|\calS_0|}\right)^2 \bE_{S,T\sim \text{Unif}(\calS)}\left[\exp\left(m\eta^2 \jiao{S,T}\right)\right] \\
&\stepe{\le} 4\bE_{S,T\sim \text{Unif}(\calS)}\left[\exp\left(m\eta^2 \left(\sum_{i=1}^d S_{ii}T_{ii}  + 2\sum_{1\le i<j\le d} S_{ij}T_{ij}\right)\right) \right] \\
&\stepf{\le} 4\exp\left( \frac{1}{2}m^2\eta^4d + 2m^2\eta^4 \binom{d}{2}\right) \le 4\exp(m^2\eta^4d^2). 
\end{align*}
Here (e) follows from Lemma \ref{lemma:S0_construction}, and (f) makes use of the independence and 1-subGaussianity of the Rademacher random variables. Piecing everything together gives that
\begin{align*}
I(\pi^\star; X^m, Y^n) = O\left(mnd^2\eta^4 + nd^3\eta^4 + 1\right) = O(mnd^2\eta^4 + 1), 
\end{align*}
where the last step is due to our assumption $m\ge d$. This proves the second upper bound in \eqref{eq:mutual_info_target}.

\section{Discussions}\label{sec:discussion}

\subsection{Other notions of minimax risk}\label{subsec:other_minimax}
Instead of using the squared Frobenius norm $\Fnorm{\Sigma^{\widehat{\pi}} - \Sigma^{\pi^\star}}^2$ over the parameter set $\{ \Sigma\succeq 0: \Opnorm{\Sigma}\le 1\}$, one may also consider a normalized Frobenius norm
\begin{align}\label{eq:normalized_Frob}
    d_{\text{NF}, \Sigma}(\pi^\star, \widehat{\pi}) := \Fnorm{ (\Sigma^{\pi^\star})^{-1/2} (\Sigma^{\widehat{\pi}} - \Sigma^{\pi^\star})(\Sigma^{\pi^\star})^{-1/2} }, 
\end{align}
as well as another parameter set based on the trace: 
\begin{align}\label{eq:Sigma_R}
    \Sigma(R) := \left\{ \Sigma \succeq 0: \trace(\Sigma)\le R \right\}, \quad R>0. 
\end{align}
Note that $d_{\text{NF}, \Sigma}$ in \eqref{eq:normalized_Frob} is closely related to the total variation distance $\text{TV}(\calN(0,\Sigma^{\widehat{\pi}}), \calN(0,\Sigma^{\pi^\star}) )$ for density estimation (cf. \cite[Theorem 1.1]{devroye2018total}), and \eqref{eq:Sigma_R} uses an $\ell_1$ ball (instead of $\ell_\infty$) of the spectrum of the covariance $\Sigma$ to promote sparsity. The following theorem summarizes two different minimax risks over $\Sigma(R)$. 

\begin{theorem}\label{thm:minimax_risks}
Let $R>0$, $d\ge 2$, and $m\ge n$ be fixed sample sizes. Then
\begin{align}\label{eq:minimax_1}
\inf_{\widehat{\pi}}\sup_{\pi^\star\in S_d, \Sigma\in \Sigma(R)} \bE_{(\pi^\star,\Sigma)}\Fnorm{\Sigma^{\widehat{\pi}} - \Sigma^{\pi^\star}}^2 \asymp \frac{R^2}{n},
\end{align}
and the GW estimator achieves this upper bound. If in addition $n\ge c_0d\log d$ and $mn\ge c_0\sqrt{d^3\log d}$ for a sufficiently large numerical constant $c_0>0$, it holds that
\begin{align}\label{eq:minimax_2}
\inf_{\widehat{\pi}}\sup_{\pi^\star\in S_d, \Sigma\in \Sigma(R)} \bE_{(\pi^\star,\Sigma)}\left[d_{\text{\rm NF}, \Sigma}(\pi^\star, \widehat{\pi})^2\right]  \asymp \frac{d\log d}{n} + \sqrt{\frac{d^3\log d}{mn}}, 
\end{align}
and the QMLE achieves this upper bound. 
\end{theorem}

Lemma \ref{thm:minimax_risks} shows that if we have the constraint on the trace instead of the operator norm, the minimax Frobenius risk in \eqref{eq:minimax_1} coincides with the minimax Frobenius risk $
    \inf_{\hS}\sup_{\Sigma} \| \hS - \Sigma\|_{\text{\rm F}}^2 $
of estimating the covariance matrix $\Sigma$ using $n$ observations. In other words, under the trace constraint, the minimax optimal estimator essentially needs to estimate the nuisance parameter $\Sigma$. This is because the hard instance of $\Sigma$ under the trace constraint has a very sparse spectrum: only $O(1)$ eigenvalues of $\Sigma$ are non-zero. This essentially reduces the dimension of the problem to $d=O(1)$, and therefore saving a $\text{poly}(d)$ factor in the minimax risk becomes vacuous. We also note that the GW estimator remains minimax rate-optimal in this scenario. 

We could circumvent the above problem by considering a different loss. By using the normalized Frobenius norm $d_{\text{NF}, \Sigma}$ in \eqref{eq:normalized_Frob}, the minimax rate in \eqref{eq:minimax_2} of Theorem \ref{thm:minimax_risks} coincides with the minimax rate in Theorem \ref{thm:upper_bound_QMLE} -- \ref{thm:lower_bound} again, even if $\Sigma$ has a sparse spectrum. Therefore, our main result in this paper is not an artifact of considering a special family of covariance matrices. 

\subsection{Comparison with orthogonal statistical learning}\label{subsec:oracle}
One particular feature of covariance alignment is that the dimension ($\asymp d^2$) of the nuisance parameter $\Sigma$ is much larger than the counterpart ($\asymp d$) of the target parameter $\pi^\star$, and the minimax rate takes a semiparametric form that interpolates between the rates for covariance estimation and permutation estimation with known covariance. This feature shares some similarities with the classical semiparametric statistics~\cite{bickel1993efficient} or more recently orthogonal statistical learning~\cite{chernozhukov2017double,foster2023orthogonal}, where a typical result is that the parametric rate can be achieved for the target parameter even if the nuisance is estimated with a slower rate. The semiparametric analysis typically relies on the Neyman orthogonality condition~\cite{neyman1959optimal} or its higher-order counterparts~\cite{mackey2018orthogonal}. Specifically, let $r_{\text{T}}$ be an achievable rate of convergence for estimating the target \emph{if the nuisance were known}, and $r_{\text{N}}$ be an achievable rate of convergence for estimating the nuisance. Then under the Neyman orthogonality condition, the rate of convergence for estimating the target under an unknown nuisance is shown \cite{foster2023orthogonal} to be $O(r_{\text{T}} + r_{\text{N}}^2)$, or even $O(r_{\text{T}} + r_{\text{N}}^4)$ if a certain strong convexity condition is met.

We illustrate our main distinction from the above line of research. First, the Neyman orthogonality condition does not hold in our problem. For example, the GW estimator uses the loss function $\ell(y; P, \Sigma) = -y^\top P \Sigma P^\top y$ in the sense that $\widehat{P}^\GW = \argmin_{P\in \text{BP}(d)} \frac{1}{n}\sum_{i=1}^n \ell(Y_i; P, \hS_X)$, while
\begin{align*}
\bE[\nabla_P \nabla_\Sigma \ell(Y; P_{\pi^\star}, \Sigma)] = \bE[ \nabla_P(-P^\top yy^\top P)|_{P = P_{\pi^\star}} ] = - \nabla_P(P^\top P_{\pi^\star}\Sigma P_{\pi^\star}^\top P)|_{P = P_{\pi^\star}} \neq 0.
\end{align*}
Second, instead of using a Taylor expansion to decouple the target estimation and nuisance estimation, we crucially make use of the specific form of the nuisance estimator when used for target estimation. This point can be rigorously established by comparing with the following oracle model. 

Under an oracle model, we still observe the sample $Y^n$; however, instead of observing the other sample $X^m$, now we have access to an oracle which outputs a covariance estimate $\widehat{\Sigma}_m$. The only property we know about $\widehat{\Sigma}_m$ is that $\|\widehat{\Sigma}_m - \Sigma\|_{\text{op}}\le \sqrt{d/m}$ almost surely. This would essentially be the case if $\widehat{\Sigma}_m$ were the empirical covariance computed from $X^m$ in the original observation model, while under the oracle model it is unknown if $\widehat{\Sigma}_m$ is unbiased and/or Wishart distributed. We note that the statistical guarantees of orthogonal statistical learning~\cite{foster2023orthogonal} only rely on such a weaker oracle model for the nuisance estimation. 

In covariance alignment, however, we prove the following theorem stating that the minimax rate of convergence for estimating $\pi^\star$ is strictly larger under the above oracle model. 

\begin{theorem}\label{thm:oracle}
For $m\ge n\ge d\log d$, under the oracle model it holds that
\begin{align*}
\inf_{\widehat{\pi}}\sup_{\pi^\star, \Sigma: \Opnorm{\Sigma}\le 1} \bE_{\Sigma} \Fnorm{\Sigma^{\widehat{\pi}} - \Sigma^{\pi^\star}}^2 \asymp \frac{d\log d}{n} + \frac{d^2}{m}. 
\end{align*}
In addition, the GW estimator (with $\hS_X$ replaced by $\widehat{\Sigma}_m$) achieves this upper bound.
\end{theorem}

By a simple AM-GM inequality, the above rate is no smaller than the minimax rate in Theorems \ref{thm:upper_bound_QMLE}-\ref{thm:lower_bound} under the original observation model. In particular, if $m=n$, the sample complexity for consistent estimation increases from $n=\Theta(\sqrt{d^3\log d})$ to $n=\Theta(d^2)$ in the oracle model. Theorem \ref{thm:oracle} shows that, the nuisance estimation error $\Theta(d^2/m)$ for $\|\hS_m - \Sigma \|_{\text{F}}^2$ is an unavoidable price to pay in the oracle model. Therefore, one cannot hope to fit the covariance alignment model using only the orthogonal statistical learning framework; instead, our proof in Section \ref{sec:upper_bounds} relies on the Wishart distribution of $\hS_X$ to prove a \emph{high-probability} curvature property of the empirical loss. We note that both the oracle model and our finding in Theorem \ref{thm:oracle} are similar to in spirit to \cite{balakrishnan2023fundamental} for several nonparametric functional estimation problems.

\paragraph*{Funding}
Yanjun Han was generously supported by the Norbert Wiener postdoctoral fellowship in statistics at MIT
IDSS. Philippe Rigollet is supported by NSF grants IIS-1838071, DMS-2022448, and CCF-2106377. George Stepaniants is supported through a National Science Foundation Graduate Research Fellowship under Grant No. 1745302.

\appendix

\section{Auxiliary lemmas}
The following lemma follows from the well-known Hanson--Wright inequality \cite{hanson1971bound,wright1973bound}. 
\begin{lemma}\label{lemma:Hanson-Wright}
Let $X_1,\cdots,X_n\in \mathbb{R}^d$ be i.i.d. random vectors, with independent zero-mean 1-subGaussian coordinates. Then there exists an absolute constant $C>0$ independent of $(n,d)$ such that, for every symmetric matrix $A\in \mathbb{R}^{d\times d}$ and $\delta\in (0,1)$,
\begin{align*}
    \mathbb{P}\left( \frac{1}{n}\sum_{i=1}^n X_i^\top A X_i \ge \bE[X^\top AX] - C\left(\sqrt{\frac{\log(1/\delta)}{n}}\Fnorm{A}+ \frac{\log(1/\delta)}{n}\Opnorm{A}\right) \right) \ge 1- \delta. 
\end{align*}
In particular, whenever $n\ge \log(1/\delta)$, it holds that
\begin{align*}
    \mathbb{P}\left( \frac{1}{n}\sum_{i=1}^n X_i^\top A X_i \ge \bE[X^\top AX] - 2C\sqrt{\frac{\log(1/\delta)}{n}}\Fnorm{A} \right) \ge 1- \delta. 
\end{align*}
\end{lemma}
\begin{proof}
For the first statement, let $X^{(n)}\in \mathbb{R}^{nd}$ be a long vector via the concatenation of $X_1,\cdots,X_n$, and $A^{(n)}\in \mathbb{R}^{nd\times nd}$ be the matrix with $n$ repeated appearances of $A/n$ in the diagonal. We have
\begin{align*}
\frac{1}{n}\sum_{i=1}^n X_i^\top AX_i = (X^{(n)})^\top A^{(n)}X^{(n)},
\end{align*}
and $\Fnorm{A^{(n)}} = \Fnorm{A}/\sqrt{n}$, $\Opnorm{A^{(n)}}=\Opnorm{A}/n$. Now the first statement follows from the standard Hanson--Wright inequality (cf. \cite[Theorem 1.1]{rudelson2013hanson}). Since $\Opnorm{A}\le \Fnorm{A}$ and $\log(1/\delta)/n\le 1$, the second statement directly follows. 
\end{proof}

The following inequality appears in \cite[Page 4]{devroye2018total}, which was stated without proof. For completeness we include a proof here. 
\begin{lemma}\label{lemma:frobinv}
    If $\Sigma \in \mathbb{R}^{d \times d}$ is a positive definite matrix then
    \begin{equation}
        \frac{1}{2} \leq \frac{\min(1, \Fnorm{\Sigma^{-1} - I})}{\min(1, \Fnorm{\Sigma - I})} \leq 2.
    \end{equation}
\end{lemma}
\begin{proof}
We only prove the upper bound, as the lower bound is entirely symmetric. Also note that the upper bound holds trivially when $\Fnorm{\Sigma - I} \ge 1/2$, so in the sequel we assume that $\Fnorm{\Sigma - I} < 1/2$. Let $\lambda_1,\cdots,\lambda_d>0$ be the eigenvalues of $\Sigma$ (with multiplicities). Since $\Fnorm{\Sigma - I} \ge \max_{i\in [d]}|\lambda_i - 1|$, we have $\lambda_i>1/2$ for all $i\in [d]$. Therefore,
\begin{align*}
\min(1, \Fnorm{\Sigma^{-1} - I})^2 &\le \Fnorm{\Sigma^{-1} - I}^2 = \sum_{i=1}^d \left(\frac{1}{\lambda_i}-1\right)^2 =\sum_{i=1}^d \frac{(\lambda_i-1)^2}{\lambda_i^2} \\
&\le 4\sum_{i=1}^d (\lambda_i-1)^2 = 4 \Fnorm{\Sigma - I}^2 = 4\min(1, \Fnorm{\Sigma - I})^2, 
\end{align*}
which is the claimed upper bound.
\end{proof}

The following lemma is a simple quadratic inequality which will be useful in the proof. 
\begin{lemma}\label{lemma:quad_bound}
    If $x \in \bR_+$ satisfies $ax^2 \le bx + c$ for $a, b,c > 0$. Then $x\le \frac{b}{a} + \sqrt{\frac{c}{a}}$. 
\end{lemma}
\begin{proof}
If $ax>b + \sqrt{ac}$, then
\begin{align*}
    ax^2 - bx -c = (ax-b)x - c > \sqrt{ac}\cdot \sqrt{\frac{c}{a}} - c = 0, 
\end{align*}
a contradiction. 
\end{proof}

\section{Deferred proofs in Section \ref{sec:upper_bounds}}
\subsection{Proof of Lemma \ref{lemma:remainder}}
We first deal with the term $R_1$. By triangle inequality, 
\begin{align}\label{eq:R1_triangle}
R_1 \le \Fnorm{ \Sigma^{1/2}(\hS_X^{-1})^{\widehat{\pi}}\Sigma^{1/2} - I_d } + \Fnorm{\Sigma^{1/2}\hS_X^{-1}\Sigma^{1/2} - I_d }. 
\end{align}

Under $\calE_3$, the second term could be easily upper bounded by
\begin{align}
\Fnorm{\Sigma^{1/2}\hS_X^{-1}\Sigma^{1/2} - I_d } &\le \sqrt{d} \Opnorm{\Sigma^{1/2}\hS_X^{-1}\Sigma^{1/2} - I_d } \nonumber \\
&\stepa{\le} \sqrt{d}\cdot \frac{\Opnorm{\Sigma^{-1/2}\hS_X\Sigma^{-1/2} - I_d }}{1-\Opnorm{\Sigma^{-1/2}\hS_X\Sigma^{-1/2} - I_d }} \nonumber \\
&\stepb{\le} 2c\sqrt{d}\cdot \sqrt{\frac{d+\log(1/\delta)}{m}}. \label{eq:R1_second_term}
\end{align}
Here (a) follows from $\Opnorm{A^{-1}-I} \le \Opnorm{A-I}/(1-\Opnorm{A-I})$ as long as $\Opnorm{A-I}<1$; the last condition is fulfilled by the event $\calE_3$ in \eqref{eq:QMLE_E3}. Step (b) directly follows from the event $\calE_3$ in \eqref{eq:QMLE_E3}. 

The upper bound of the first term is slightly more involved. By triangle inequality, 
\begin{align}
   \Fnorm{ \Sigma^{1/2}(\hS_X^{-1})^{\widehat{\pi}}\Sigma^{1/2} - I_d }   &= \Fnorm{(\Sigma^{1/2}P_{\widehat{\pi}}\Sigma^{-1/2})(\Sigma^{1/2}\hS_X^{-1}\Sigma^{1/2})(\Sigma^{-1/2}P_{\widehat{\pi}}^\top \Sigma^{1/2}) - I_d } \nonumber \\
    &\leq \Fnorm{(\Sigma^{1/2}P_{\widehat{\pi}}\Sigma^{-1/2})(\Sigma^{-1/2}P_{\widehat{\pi}}^\top \Sigma^{1/2}) - I_d } \nonumber \\
    &\qquad+ \Fnorm{(\Sigma^{1/2}P_{\widehat{\pi}}\Sigma^{-1/2})(\Sigma^{1/2}\hS_X^{-1}\Sigma^{1/2}-I_d)(\Sigma^{-1/2}P_{\widehat{\pi}}^\top \Sigma^{1/2})} \nonumber \\
    &\stepc{\leq} \Fnorm{ \Sigma^{1/2}(\Sigma^{-1})^{\widehat{\pi}}\Sigma^{1/2}  - I_d } + \Opnorm{\Sigma^{1/2}P_{\widehat{\pi}}\Sigma^{-1/2}}^2 \Fnorm{\Sigma^{1/2}\hS_X^{-1}\Sigma^{1/2} - I_d} \nonumber \\
    &\stepd{=} \Fnorm{ \Sigma^{1/2}(\Sigma^{-1})^{\widehat{\pi}}\Sigma^{1/2}  - I_d } + \Fnorm{\Sigma^{1/2}\hS_X^{-1}\Sigma^{1/2} - I_d}. \label{eq:R1_first_term}
\end{align}
Here (c) uses $\Fnorm{ABC}\le \Opnorm{A}\Fnorm{BC}\le \Opnorm{A}\Fnorm{B}\Opnorm{C}$, and (d) uses the similarity between $\Sigma^{1/2}P_{\widehat{\pi}}\Sigma^{-1/2}$ and $P_{\widehat{\pi}}$ to conclude that $\Opnorm{\Sigma^{1/2}P_{\widehat{\pi}}\Sigma^{-1/2}} = 1$. 

A combination of \eqref{eq:R1_triangle}, \eqref{eq:R1_second_term}, and \eqref{eq:R1_first_term}
gives that
\begin{align*}
R_1 \le \Fnorm{ \Sigma^{1/2}(\Sigma^{-1})^{\widehat{\pi}}\Sigma^{1/2}  - I_d } + 4c\sqrt{d}\cdot \sqrt{\frac{d+\log(1/\delta)}{m}}, 
\end{align*}
which is the first statement of the lemma.  

As for $R_2$, we begin with an upper bound on the operator norm of 
\begin{align*}
\Sigma^{1/2}(\hS_X^{-1} - \Sigma^{-1} - \Sigma^{-1}(\Sigma - \hS_X)\Sigma^{-1})\Sigma^{1/2} =: A^{-1} + A - 2I_d,
\end{align*}
where $A := \Sigma^{-1/2}\hS_X \Sigma^{-1/2}$. By the definition of $\calE_3$ in \eqref{eq:QMLE_E3}, all eigenvalues of $A$ lie in the interval $[1-\eta, 1+\eta]$, with
\begin{align*}
    \eta = \min\left\{c\sqrt{\frac{d+\log(1/\delta)}{m}}, \frac{1}{2} \right\}. 
\end{align*}
It is clear that all eigenvalues of $A^{-1}+A-2I_d$ take the form $\lambda^{-1}+\lambda - 2$, with $\lambda$ being an eigenvalue of $A$. For $\lambda \in [1-\eta,1+\eta]$, we have
\begin{align*}
    \frac{1}{\lambda} + \lambda - 2 = \frac{(\lambda-1)^2}{\lambda} \le \frac{\eta^2}{1-\eta}\le 2c^2\cdot \frac{d+\log(1/\delta)}{m}. 
\end{align*}
Therefore, we have established that
\begin{align}\label{eq:R2_operatornorm}
    \Opnorm{\Sigma^{1/2}(\hS_X^{-1} - \Sigma^{-1} - \Sigma^{-1}(\Sigma - \hS_X)\Sigma^{-1})\Sigma^{1/2}} \le 2c^2\cdot \frac{d+\log(1/\delta)}{m}. 
\end{align}
Consequently, 
\begin{align*}
|R_2| &= |\jiao{ \Sigma^{-1/2}(\Sigma^{\widehat{\pi}^{-1} } - \Sigma)\Sigma^{-1/2}, \Sigma^{1/2}(\hS_X^{-1} - \Sigma^{-1} - \Sigma^{-1}(\Sigma - \hS_X)\Sigma^{-1})\Sigma^{1/2}}| \\
&\le \Fnorm{ \Sigma^{-1/2}(\Sigma^{\widehat{\pi}^{-1} } - \Sigma)\Sigma^{-1/2}}\cdot \Fnorm{\Sigma^{1/2}(\hS_X^{-1} - \Sigma^{-1} - \Sigma^{-1}(\Sigma - \hS_X)\Sigma^{-1})\Sigma^{1/2}} \\
&\le \Fnorm{ \Sigma^{-1/2}(\Sigma^{\widehat{\pi}^{-1} } - \Sigma)\Sigma^{-1/2}}\cdot \sqrt{d} \Opnorm{\Sigma^{1/2}(\hS_X^{-1} - \Sigma^{-1} - \Sigma^{-1}(\Sigma - \hS_X)\Sigma^{-1})\Sigma^{1/2}} \\
&\overset{\eqref{eq:R2_operatornorm}}{\le} 2c^2\cdot \frac{\sqrt{d}(d+\log(1/\delta))}{m}\Fnorm{\Sigma^{-1/2}\Sigma^{\widehat{\pi}^{-1} }\Sigma^{-1/2} - I_d}. 
\end{align*}
This completes the proof of the upper bound on $|R_2|$. \qed

\subsection{Proof of Lemma \ref{lemma:L1_L2}}
Let $S\triangleq \sum_{i=1}^d(x_i-1)$. By AM-GM inequality we have $S\ge 0$.

We first prove an upper bound on $\max_i x_i$ in terms of $S$. For each $i\in [d]$, it holds that
\begin{align*}
    S + d = x_i + \sum_{j\neq i} x_j \ge x_i + (d-1)\left(\prod_{j\neq i} x_j\right)^{\frac{1}{d-1}} = x_i + \frac{d-1}{x_i^{1/(d-1)}}. 
\end{align*}
As $(1/x_i)^{1/(d-1)} = e^{-\log(x_i)/(d-1)} \ge 1 - \log(x_i)/(d-1)$, the above inequality gives
\begin{align*}
    S + 1 \ge x_i - \log x_i.
\end{align*}
As the function $x\mapsto h(x)\triangleq x-\log x$ is increasing on $[1,\infty)$, and $h(2S+2) = 2S + 2 - \log(2) - \log(S+1) \ge 2S + 2 - 1 - S = S+1$, we conclude that $x_i\le 2(S+1)$. This inequality holds for all $i\in [d]$, so $\max_{i} x_i \le 2(S+1)$. 

Next we introduce the following lemma, the proof of which is deferred to the end of this section. 

\begin{lemma}\label{lemma:logbound}
    For $x \in (-1, r]$ we have that
    \begin{align*}
        \log(1 + x) \leq x - \frac{x^2}{2(r+1)}.
    \end{align*}
\end{lemma}

Applying Lemma \ref{lemma:logbound} with $r=2S+1$ gives  
\begin{align*}
0 = \sum_{i=1}^d \log(x_i) \le \sum_{i=1}^d\left( (x_i-1) - \frac{(x_i-1)^2}{4(S+1)}\right) = S - \frac{1}{4(S+1)}\sum_{i=1}^d (x_i-1)^2, 
\end{align*}
which implies that
\begin{align*}
\sum_{i=1}^d (x_i-1)^2 \le 4S(S+1),
\end{align*}
as claimed. \qed 

\begin{proof}[Proof of Lemma \ref{lemma:logbound}]
Define
    \begin{align*}
        f(x) = x - \frac{x^2}{2(r+1)} - \log(1 + x). 
    \end{align*}
We have
    \begin{align*}
        f'(x) = 1 - \frac{x}{r+1} - \frac{1}{x+1} = \frac{x(r-x)}{(x+1)(r+1)}, 
    \end{align*}
which satisfies $f'(x)\ge 0$ for $x\in [0,r]$ and $f'(x)\le 0$ for $x\in (-1,0]$. Hence, $f(x)\ge f(0)=0$ for all $x\in (-1,r]$, as claimed. 
\end{proof}

\section{Deferred proofs in Section \ref{sec:lower_bound}}
\subsection{Proof of Lemma \ref{lemma:S0_construction}}
By the probabilistic argument, it suffices to show that by drawing a random matrix $S\sim \text{Unif}(\calS)$, the probability that $S$ satisfies all three properties is at least $1/2$. The first property holds trivially. For the second property, it was shown in \cite[Corollary 4.4.8]{vershynin2018high} that
\begin{align*}
\mathbb{P}(\Opnorm{S} \ge c_1\sqrt{d}) \le 4\exp(-c_1d)
\end{align*}
holds for some absolute constant $c_1>0$ independent of $d$. By choosing $d_0>0$ large enough we may ensure that $\mathbb{P}(\Opnorm{S} \ge c_1\sqrt{d})\le 1/4$ for all $d\ge d_0$.  

For the last property, note that for fixed $\pi_1, \pi_2\in S_d$, we have
\begin{align*}
\Fnorm{S^{\pi_1} - S^{\pi_2}}^2 = \Fnorm{P_{\pi_1}S P_{\pi_1}^\top - P_{\pi_2}S P_{\pi_2}^\top}^2 = \Fnorm{P_{\pi_2^{-1}}P_{\pi_1}S P_{\pi_1}^\top P_{\pi_2^{-1}}^\top - S }^2 = \Fnorm{P_{\pi_1\circ \pi_2^{-1}}S P_{\pi_1\circ \pi_2^{-1}}^\top - S }^2, 
\end{align*}
where we have used that $P_{\pi}^{-1} = P_{\pi^{-1}}$ and $P_{\pi_1}P_{\pi_2}=P_{\pi_2\circ \pi_1}$. Let $\pi = \pi_1\circ \pi_2^{-1}\in S_d$. Note that
\begin{align*}
\Fnorm{S^\pi - S}^2 = \sum_{i,j=1}^d (S_{\pi(i),\pi(j)}-S_{i,j})^2 =: X^\top AX, 
\end{align*}
where $X$ is the $\binom{d+1}{2}$-dimensional column vector $(S_{i,j})_{1\le i\le j\le d}$, and the matrix $A$ is given as follows: for $i\le j$ and $k\le \ell$, 
\begin{align*}
A_{(i,j),(k,\ell)} = \begin{cases}
    2\cdot \mathbbm{1}(i=k) - \mathbbm{1}(\pi(i)= k) - \mathbbm{1}(\pi(k)=i), &\text{if } i=j, k=\ell; \\
    0, & \text{if } i=j, k\neq \ell;\\
    4[1-\mathbbm{1}((\pi(i),\pi(j)) = (i,j) \text{ or }(j,i))], & \text{if } i\neq j, (k,\ell) = (i,j); \\
    -2[\mathbbm{1}((\pi(i),\pi(j)) = (k,\ell) \text{ or }(\ell,k)) \\ 
    \qquad + \mathbbm{1}((\pi(k),\pi(\ell)) = (i,j) \text{ or }(j,i))], & \text{if } i\neq j, (k,\ell)\neq (i,j).
\end{cases}
\end{align*}
Clearly $A$ is symmetric, and
\begin{align*}
    &\max_{1\le i\le j\le d}\sum_{1\le k\le \ell \le d} |A_{(i,j),(k,\ell)}| \\
    &\le   \max_{1\le i\le j\le d} \sum_{1\le k\le \ell \le d}4\cdot \mathbbm{1}(  \{k,\ell\} = \{i,j\}, \{\pi(i),\pi(j)\} \text{ or } \{\pi^{-1}(i), \pi^{-1}(j) \} ) \\
    &\le 12. 
\end{align*}
As a result, we have
\begin{align}
\Opnorm{A} &\le \max_{1\le i\le j\le d}\sum_{1\le k\le \ell \le d} |A_{(i,j),(k,\ell)}| \le 12, \label{eq:A_opnorm} \\
\Fnorm{A}^2 &\le 4 \sum_{1\le i\le j\le d}\sum_{1\le k\le \ell \le d} |A_{(i,j),(k,\ell)}| \le 48\binom{d+1}{2} = 24d(d+1). \label{eq:A_Fnorm}
\end{align}
In addition, whenever $\sum_{i=1}^d \mathbbm{1}(\pi(i)\neq i) = \sum_{i=1}^d \mathbbm{1}(\pi_1(i)\neq \pi_2(i))\ge d/10$, it holds that
\begin{align}\label{eq:A_trace}
\trace(A) = \sum_{1\le i\le j\le d} A_{(i,j),(i,j)} &= 2\sum_{i,j=1}^d \mathbbm{1}((\pi(i),\pi(j))\neq (i,j), (j,i)) \nonumber\\
&\ge 2\sum_{i,j=1}^d \mathbbm{1}(\pi(i) \neq i )\mathbbm{1}(\pi(i) \neq j)\nonumber \\
&= 2(d-1)\sum_{i=1}^d \mathbbm{1}(\pi(i)\neq i) \ge \frac{d^2}{10}. 
\end{align}

With the help of \eqref{eq:A_opnorm}-\eqref{eq:A_trace}, we are ready to prove a high-probability lower bound of $\Fnorm{S^{\pi_1} - S^{\pi_2}}^2 = X^\top AX$. Applying the Hanson--Wright inequality (cf. Lemma \ref{lemma:Hanson-Wright}) with $n=1$ to $X^\top AX$, with probability exceeding $1-1/[4(d!)^2]$ it holds that
\begin{align*}
\Fnorm{S^{\pi_1} - S^{\pi_2}}^2 = X^\top AX &\ge \bE[X^\top AX] - C\left(\sqrt{d\log d}\Fnorm{A} + (d\log d)\Opnorm{A}\right) \\
&= \trace(A) - C\left(\sqrt{d\log d}\Fnorm{A} + (d\log d)\Opnorm{A}\right) \\
&\ge \frac{d^2}{10} - O\left(\sqrt{d\log d}\cdot d + d\log d\right) \ge c_2d^2
\end{align*}
as long as $d\ge d_0$ for a sufficiently large $d_0$ and a sufficiently small $c_2>0$. Now applying the union bound over at most $(d!)^2$ choices of $(\pi_1, \pi_2)$, we conclude that $S$ satisfies the third property with probability at least $3/4$. 

Finally, by a union bound again, with probability at least $1/2$, the random matrix $S$ satisfies all three properties. This completes the proof of the lemma. \qed 

\subsection{Proof of Lemma \ref{lemma:Fano}}\label{subsec:Fano_proof}
Fix any estimator $T$. Let $\mathbb{P}$ be the joint distribution of $(\theta,X)\sim \pi(\theta)P_\theta(x)$ under the Bayes setup, and $\mathbb{Q}$ be another joint distribution of $(\theta,X)\sim \pi(\theta)Q(x)$, where $\theta$ and $X$ are independent. Here $Q$ is an arbitrary probability distribution over $\calX$. Consider a map $\Phi: \Theta\times \calX \to \{0,1\}$ defined as $\Phi(\theta,x)=\mathbbm{1}(L(\theta,T(X)) \le \Delta)$, the data-processing inequality gives that
\begin{align*}
    D_{\text{KL}}(\mathbb{P}\|\mathbb{Q}) &\ge D_{\text{KL}}(\Phi_\# \mathbb{P} \| \Phi_\# \mathbb{Q}) \\
    &= \mathbb{P}(L(\theta,T(X)) \le \Delta)\log \frac{\mathbb{P}(L(\theta,T(X)) \le \Delta)}{\mathbb{Q}(L(\theta,T(X)) \le \Delta)} \\
    &\qquad + (1-\mathbb{P}(L(\theta,T(X))) \le \Delta)\log \frac{1-\mathbb{P}(L(\theta,T(X)) \le \Delta)}{1-\mathbb{Q}(L(\theta,T(X)) \le \Delta)} \\
    &\ge \mathbb{P}(L(\theta,T(X)) \le \Delta)\log \frac{1}{\mathbb{Q}(L(\theta,T(X)) \le \Delta)} -\log 2, 
\end{align*}
where the last step is due to $x\log x+(1-x)\log(1-x)\ge -\log 2$ for all $x\in [0,1]$. Since
\begin{align*}
\mathbb{Q}(L(\theta,T(X)) \le \Delta) \le \sup_{a} \pi\{ \theta\in \Theta: L(\theta,a)\le \Delta\} = p_\Delta,
\end{align*}
the above inequality rearranges to
\begin{align*}
\bE_\mathbb{P}[L(\theta,T(X)] &\ge \Delta\cdot \mathbb{P}(L(\theta,T(X)) > \Delta) \\
&\ge \Delta\left(1 - \frac{D_{\text{KL}}(\mathbb{P}\|\mathbb{Q}) + \log 2}{\log(1/p_\Delta)}\right). 
\end{align*}
Finally, since the above inequality holds for all $Q$, using the variational representation of the mutual information $I(\theta;X) = \min_Q D_{\text{KL}}(\pi(\theta)P_\theta(x) \| \pi(\theta)Q(x))$ completes the proof. \qed 

\subsection{Proof of Lemma \ref{lemma:p_Delta} }
By the third property of $\calS_0$, we have the inclusion
\begin{align*}
\left\{ \pi^\star: L(\pi^\star, \widehat{\pi}) \le \frac{c_2\eta^2d^2}{4} \right\}  = \left\{ \pi^\star: \Fnorm{S^{\pi^\star} - S^{\widehat{\pi}}}^2 \le c_2d^2 \right\} \subseteq  \left\{ \pi^\star: \sum_{i=1}^d \mathbbm{1}(\pi^\star(i)\neq \widehat{\pi}(i)) \le \frac{d}{10} \right\}. 
\end{align*}
As $\pi^\star\sim \text{Unif}(S_d)$, for any fixed $\widehat{\pi}$, the final probability is at most
\begin{align*}
    \frac{1}{d!} \sum_{k=0}^{\lfloor d/10 \rfloor} \binom{d}{k} k! = \sum_{k=0}^{\lfloor d/10 \rfloor}
\frac{1}{(d-k)!} \le \left(\frac{d}{10}+1\right)\frac{1}{\lceil 9d/10 \rceil ! } \le \exp(-c_3d\log d)
\end{align*}
for some absolute constant $c_3>0$. This shows that $p_\Delta \le d^{-c_3d}$. \qed 

\subsection{Proof of Lemma \ref{lemma:simple_upperbound}}
The following upper bound holds for $I(\pi^\star; X^m, Y^n)$:
\begin{align*}
I(\pi^\star; X^m, Y^n) &= I(\pi^\star; X^m) + I(\pi^\star; Y^n\mid X^m) \\
&\stepa{=} I(\pi^\star; Y^n\mid X^m) \\
&\stepb{\le} \bE_{X^m, \pi^\star}[D_{\text{KL}}( P_{Y^n\mid \pi^\star, X^m} \| \calN(0,I_d/2)^{\otimes n} ) ] \\
&\stepc{\le} \bE_{\Sigma,\pi^\star}[D_{\text{KL}}( P_{Y^n\mid \pi^\star, \Sigma} \| \calN(0,I_d/2)^{\otimes n} ) ] \\
&\stepd{=} \frac{n}{2}\bE_{S}[\eta \trace(S) - \log\det(I_d + \eta S)]. 
\end{align*}
Here (a) is due to the independence between $\pi^\star$ and $X^m$, (b) follows from the variational representation of the mutual information, (c) follows from the convexity of the KL divergence and that $P_{Y^n\mid \pi^\star,X^m} = \bE_{\Sigma\mid X^m} [P_{Y^n\mid \pi^\star,\Sigma}]$, (d) is the KL divergence between two Gaussians as well as the simple facts $\trace(S)=\trace(S^{\pi^\star}), \det(I_d+\eta S) = \det(I_d+\eta S^{\pi^\star})$. To proceed, let $\lambda_1,\cdots,\lambda_d$ be the eigenvalues of $S$, so that
\begin{align*}
\eta \trace(S) - \log\det(I_d + \eta S) &= \sum_{i=1}^d \left( \eta \lambda_i - \log(1+\eta \lambda_i) \right) \\
&\stepe{\le} \sum_{i=1}^d \eta^2\lambda_i^2 = \eta^2 \Fnorm{S}^2 = \eta^2 d^2. 
\end{align*}
Here (e) follows from $x-\log(1+x)\le x^2$ whenever $|x|\le 1/2$, and we recall that $\eta \le 1/(2c_1\sqrt{d})$ from \eqref{eq:eta_range}, and that $|\lambda_i|\le c_1\sqrt{d}$ almost surely from the second property of $\calS_0$. A combination of the above inequalities gives that $I(\pi^\star; X^m, Y^n)\le n\eta^2 d^2/2$, as desired. \qed 

\subsection{Proof of Lemma \ref{lemma:superadditivity}}
For the first statement, note that
\begin{align*}
    a_{n+1} - a_n &= D_{\text{KL}}(P_{Z^{n+1}}\| Q^{\otimes (n+1)}) - D_{\text{KL}}(P_{Z^{n}}\| Q^{\otimes n}) \\
    &\stepa{=} \bE_{P_{Z^n}}[D_{\text{KL}}(P_{Z_{n+1}\mid Z^n} \| Q)] \\
    &\stepb{\ge} \bE_{P_{Z_2^n}}[D_{\text{KL}}(\bE_{P_{Z_1\mid Z_2^n}}[P_{Z_{n+1}\mid Z^n}] \| Q)] \\
    &= \bE_{P_{Z_2^n}}[D_{\text{KL}}(P_{Z_{n+1}\mid Z_2^n} \| Q)] \\
    &\stepc{=} \bE_{P_{Z^{n-1} }}[D_{\text{KL}}(P_{Z_{n}\mid Z^{n-1}} \| Q)] \\
    &= a_n - a_{n-1}, 
\end{align*}
where (a) is crucially thanks to the product structure on $Q$ and could be verified by simple algebra, (b) follows from the convexity of the KL divergence, and (c) is due to the stationarity of $(Z_1,Z_2,\cdots)$. The first statement is therefore proved. The second statement directly follows from
\begin{align*}
    \frac{a_m}{m} = \frac{\sum_{k=1}^m (a_k-a_{k-1})}{m} \ge  \frac{\sum_{k=1}^n (a_k-a_{k-1})}{n} = \frac{a_n}{n}.  
\end{align*}\qed

\subsection{Proof of Lemma \ref{lemma:second_moment}}\label{subsec:second_moment_proof}
As a shorthand we denote the centered Gaussian distribution $\calN(0,\Sigma_S)$ by $\mathbb{P}_S$, and $\calN(0,I_d)$ by $\mathbb{Q}$. The standard second moment computation gives
\begin{align*}
\begin{aligned}
    \chi^2(\bE_{S}[\calN(0,\Sigma_S)^{\otimes p}] \| \calN(0, I_d)^{\otimes p}) + 1 &= \mathbb{E}_{\mathbb{Q}^{\otimes p}}\Big[\Big(\frac{d\bE_{S}[\mathbb{P}_S^{\otimes p}]}{d\mathbb{Q}^{\otimes p}}\Big)^2\Big] \\
    &= \bE_{S, T}\Bigg[\mathbb{E}_{\mathbb{Q}^{\otimes p}}\Big[\frac{d\mathbb{P}_S^{\otimes p}}{d\mathbb{Q}^{\otimes p}}\frac{d\mathbb{P}_T^{\otimes p}}{d\mathbb{Q}^{\otimes p}}\Big]\Bigg] \\
    &= \bE_{S, T}\Big[\Big(\mathbb{E}_{\mathbb{Q}}\Big[\frac{d\mathbb{P}_S}{d\mathbb{Q}}\frac{d\mathbb{P}_T}{d\mathbb{Q}}\Big]\Big)^p\Big] .
\end{aligned}
\end{align*}

By Cauchy-Schwarz, 
\begin{align*}
\mathbb{E}_{\mathbb{Q}}\Big[\frac{d\mathbb{P}_S}{d\mathbb{Q}}\frac{d\mathbb{P}_T}{d\mathbb{Q}}\Big] &\le \left(\mathbb{E}_{\mathbb{Q}}\Big[\Big(\frac{d\mathbb{P}_S}{d\mathbb{Q}}\Big)^2\Big]\right)^{\frac{1}{2}} \left(\mathbb{E}_{\mathbb{Q}}\Big[\Big(\frac{d\mathbb{P}_T}{d\mathbb{Q}}\Big)^2\Big]\right)^{\frac{1}{2}}  =(\chi^2(\mathbb{P}_S\| \mathbb{Q})+1)^{1/2} (\chi^2(\mathbb{P}_T\| \mathbb{Q})+1)^{1/2}, 
\end{align*}
and 
\begin{align*}
\chi^2(\mathbb{P}_S\| \mathbb{Q})+1 = \frac{\det(2\Sigma_S^{-1} - I_d)}{\det(\Sigma_S)} < \infty
\end{align*}
holds as long as $\Sigma_S\succ 0$ and $2\Sigma_S^{-1} - I_d \succ 0$. This is ensured by our assumption $0\prec \Sigma_S \prec 2I_d$, so the above expression is finite. Now note that
\begin{align*}
    \mathbb{E}_{\mathbb{Q}}\Big[\frac{d\mathbb{P}_S}{d\mathbb{Q}}\frac{d\mathbb{P}_T}{d\mathbb{Q}}\Big] &= \frac{1}{\sqrt{\det(\Sigma_S)\det(\Sigma_T)}}\mathbb{E}_{X\sim \mathbb{Q}}\Big[\exp\Big(-\frac{1}{2}X^\top\Big(\Sigma_S^{-1} + \Sigma_T^{-1} - 2I_d\Big)X\Big)\Big]\\
    &= \frac{1}{\sqrt{\det(\Sigma_S)\det(\Sigma_T)}}\mathbb{E}_{Z\sim \mathbb{Q}}\Big[\exp\Big(Z^\top AZ\Big)\Big]
\end{align*}
where $A := I_d - (\Sigma_S^{-1} + \Sigma_T^{-1})/2$. Let $\lambda_1,\cdots,\lambda_d$ be the eigenvalues of $A$. Since $0\prec \Sigma_S, \Sigma_T \prec 2I_d$, we have $\lambda_k<1/2$ for all $k\in [d]$. Therefore, by diagonalization, 
\begin{align*}
    \mathbb{E}_{Z\sim \mathbb{Q}}\Big[\exp\Big(Z^\top AZ\Big)\Big] = \prod_{k=1}^d\bE_{Z_k\sim \cN(0,1)}\Big[\exp\Big(\lambda_kZ_k^2\Big)\Big] = \prod_{k=1}^d \frac{1}{(1-2\lambda_k)^{1/2}}. 
\end{align*}

Hence, we can write
\begin{align*}
    \mathbb{E}_{\mathbb{Q}}\Big[\frac{d\mathbb{P}_S}{d\mathbb{Q}}\frac{d\mathbb{P}_T}{d\mathbb{Q}}\Big]  &= \frac{\det(I_d - 2A)^{-\frac{1}{2}}}{\sqrt{\det(\Sigma_S)\det(\Sigma_T)}} = \frac{\det(\Sigma_S^{-1} + \Sigma_T^{-1} - I_d)^{-\frac{1}{2}}}{\sqrt{\det(\Sigma_S)\det(\Sigma_T)}}\\
    &= \Big( \det(\Sigma_S)\det(\Sigma_S^{-1} + \Sigma_T^{-1} - I_d)\det(\Sigma_T) \Big)^{-1/2}\\
    &= \det(\Sigma_S + \Sigma_T - \Sigma_S  \Sigma_T)^{-1/2} \\
    &= \det( I_d - (\Sigma_S - I_d)(\Sigma_T - I_d) ) ^{-1/2}. 
\end{align*}
Finally, this proves that
\begin{align*}
    \chi^2(\bE_{S}[\calN(0,\Sigma_S)^{\otimes p}] \| \calN(0, I_d)^{\otimes p}) + 1 = \bE_{S, T}\Big[\det\Big(I_d - (\Sigma_S - I_d)(\Sigma_T - I_d)\Big)^{-\frac{p}{2}}\Big], 
\end{align*}
which is the desired result.\qed

\section{Deferred proofs in Section \ref{sec:discussion}}
\subsection{Proof of Theorem \ref{thm:minimax_risks}}
We establish \eqref{eq:minimax_1} and \eqref{eq:minimax_2} separately. 
\subsubsection{Proof of \eqref{eq:minimax_1}}
For the upper bound, we first note that
\begin{align*}
\bE\left[\Fnorm{\hS_X  - \Sigma}^2\right] &= \sum_{i,j=1}^d \bE[ ((\hS_X)_{i,j} - \Sigma_{i,j})^2 ] \\
&= \frac{1}{m}\sum_{i,j=1}^d \bE[ (X_{1,i}X_{1,j} - \Sigma_{i,j})^2 ] \\
&\stepa{=} \frac{1}{m}\sum_{i,j=1}^d \left( \bE[X_{1,i}^2 X_{1,j}^2] - \Sigma_{i,j}^2 \right) \\
&\stepb{=} \frac{1}{m}\sum_{i,j=1}^d \left( \bE[X_{1,i}^2] \bE[X_{1,j}^2] + 2 \bE[X_{1,i}X_{1,j}]^2 - \Sigma_{i,j}^2 \right) \\
&= \frac{1}{m}\sum_{i,j=1}^d \left( \Sigma_{i,i}\Sigma_{j,j} + \Sigma_{i,j}^2 \right) \\
&\stepc{\le} \frac{2}{m} \sum_{i,j=1}^d \Sigma_{i,i}\Sigma_{j,j} \\
&= \frac{2\trace(\Sigma)^2}{m} \le \frac{2R^2}{m}. 
\end{align*}
Here (a) is due to the unbiasedness $\bE[X_{1,i}X_{1,j}]=\Sigma_{i,j}$, (b) follows from Isserlis' theorem \cite{isserlis1918formula}, and (c) uses $\Sigma_{i,j}^2 \le \Sigma_{i,i}\Sigma_{j,j}$ for any PSD matrix $\Sigma$. Similarly, we have
\begin{align*}
\bE\left[\Fnorm{\hS_Y  - \Sigma^{\pi^\star}}^2\right] \le \frac{2R^2}{n}. 
\end{align*}

For the estimator $\widehat{\pi}$, we simply choose $\widehat{\pi} = \widehat{\pi}^\GW$ to be the GW estimator defined in \eqref{eq:GW_estimator}. Then
\begin{align*}
\Fnorm{\Sigma^{\widehat{\pi}} - \Sigma^{\pi^\star} } &\le \Fnorm{\Sigma^{\widehat{\pi}} - \hS_X^{\widehat{\pi}} } + \Fnorm{\hS_X^{\widehat{\pi}} - \hS_Y } + \Fnorm{\hS_Y - \Sigma^{\pi^\star}} \\
&\stepd{\le} \Fnorm{\Sigma^{\widehat{\pi}} - \hS_X^{\widehat{\pi}} } + \Fnorm{\hS_X^{\pi^\star} - \hS_Y } + \Fnorm{\hS_Y - \Sigma^{\pi^\star}} \\
&\le \Fnorm{\Sigma^{\widehat{\pi}} - \hS_X^{\widehat{\pi}} } + \Fnorm{\hS_X^{\pi^\star} - \Sigma^{\pi^\star} } + 2 \Fnorm{\hS_Y - \Sigma^{\pi^\star}} \\
&= 2\Fnorm{\hS_X  - \Sigma} + 2\Fnorm{\hS_Y  - \Sigma^{\pi^\star}}, 
\end{align*}
where (d) follows from the defining property of the GW estimator in \eqref{eq:GW_estimator}. Combining the above displays shows the upper bound
\begin{align*}
\bE\left[\Fnorm{\Sigma^{\widehat{\pi}} - \Sigma^{\pi^\star} }^2\right] = O\left(\frac{R^2}{n}\right). 
\end{align*}

For the lower bound in \eqref{eq:minimax_1}, by restricting the support of $\Sigma$ within its upper left $2\times 2$ corner, it is clear that the minimax risk is no smaller than its counterpart with $d=2$. In the sequel we assume that $d=2$ and lower bound the minimax risk as
\begin{align*}
\inf_{\widehat{\pi}}\sup_{\pi^\star\in S_d, \Sigma\in \Sigma(R)} \bE_{(\pi^\star,\Sigma)}\Fnorm{\Sigma^{\widehat{\pi}} - \Sigma^{\pi^\star}}^2 &\ge \inf_{\widehat{\pi}}\sup_{(\pi^\star, \Sigma): \Opnorm{\Sigma}\le R/2 } \bE_{(\pi^\star,\Sigma)}\Fnorm{\Sigma^{\widehat{\pi}} - \Sigma^{\pi^\star}}^2 \\
&\stepe{=} \frac{R^2}{4}\inf_{\widehat{\pi}}\sup_{(\pi^\star, \Sigma): \Opnorm{\Sigma}\le 1 } \bE_{(\pi^\star,\Sigma)}\Fnorm{\Sigma^{\widehat{\pi}} - \Sigma^{\pi^\star}}^2  \\
&\stepf{=} \Omega\left(\frac{R^2}{n}\right), 
\end{align*}
where (e) follows from simple scaling, and (f) is the standard $\Omega(1/n)$ parametric rate (alternatively, this also follows from the $\Omega(d\log d/n)$ lower bound established in Section \ref{sec:lower_bound} applied to $d=2$). This completes the proof of \eqref{eq:minimax_1}. 

\subsubsection{Proof of \eqref{eq:minimax_2}}
For the upper bound, note that the analysis of $\widehat{\pi} = \widehat{\pi}^\QMLE$ in Section \ref{sec:QMLE} holds for any $\Sigma\succeq 0$ until \eqref{eq:QMLE_trace_upperbound}. Assuming $\pi^\star=\id$, integrating the tail over $\delta\in (0,1)$ in \eqref{eq:QMLE_trace_upperbound} gives that
\begin{align*}
\bE[\jiao{ \Sigma^{\widehat{\pi}^{-1}} - \Sigma, \Sigma^{-1} }] = O\left(\frac{d\log d}{n} + \sqrt{\frac{d^3 \log d}{mn}} \right).
\end{align*}
By the second inequality in Corollary \ref{cor:trace_frob}, we have
\begin{align}\label{eq:normalized_Frob_upper}
\Fnorm{\Sigma^{1/2}(\Sigma^{-1})^{\widehat{\pi}}\Sigma^{1/2} - I_d }^2 \le C\left(\frac{d\log d}{n} + \sqrt{\frac{d^3 \log d}{mn}} \right), 
\end{align}
with $C>0$ being an absolute constant independent of $(n,m,d)$. Since $n\ge c_0d\log d$ and $mn\ge c_0d^3\log d$, for a large enough numerical constant $c_0>0$ depending only on $C$, the RHS of \eqref{eq:normalized_Frob_upper} will be smaller than $1/2$. By Lemma \ref{lemma:frobinv}, 
\begin{align*}
\min\left\{1, \Fnorm{\Sigma^{-1/2}\Sigma^{\widehat{\pi}}\Sigma^{-1/2} - I_d }^2 \right\} &\le 2 \min\left\{1, \Fnorm{\Sigma^{1/2}(\Sigma^{-1})^{\widehat{\pi}}\Sigma^{1/2} - I_d }^2 \right\} \\
&\le 2C\left(\frac{d\log d}{n} + \sqrt{\frac{d^3 \log d}{mn}} \right), 
\end{align*}
where the final quantity is smaller than $1$. This gives
\begin{align*}
\Fnorm{\Sigma^{-1/2}(\Sigma^{\widehat{\pi}} - \Sigma)\Sigma^{-1/2}}^2 = O\left(\frac{d\log d}{n} + \sqrt{\frac{d^3 \log d}{mn}} \right), 
\end{align*}
which is the desired upper bound. 

The minimax lower bound is an easy consequence of Theorem \ref{thm:lower_bound}: 
\begin{align*}
\inf_{\widehat{\pi}}\sup_{\pi^\star\in S_d, \Sigma\in \Sigma(R)} \bE_{(\pi^\star,\Sigma)}[d_{\text{\rm NF}, \Sigma}(\pi^\star, \widehat{\pi})^2] &\ge \inf_{\widehat{\pi}}\sup_{(\pi^\star,\Sigma): \Opnorm{\Sigma}\le R/d} \bE_{(\pi^\star,\Sigma)}[d_{\text{\rm NF}, \Sigma}(\pi^\star, \widehat{\pi})^2] \\
&\stepa{=} \inf_{\widehat{\pi}}\sup_{(\pi^\star,\Sigma): \Opnorm{\Sigma}\le 1} \bE_{(\pi^\star,\Sigma)}[d_{\text{\rm NF}, \Sigma}(\pi^\star, \widehat{\pi})^2]  \\
&\stepb{\ge} \inf_{\widehat{\pi}}\sup_{(\pi^\star,\Sigma): \Opnorm{\Sigma}\le 1} \bE_{(\pi^\star,\Sigma)}\left[\Fnorm{\Sigma^{\widehat{\pi}} - \Sigma^{\pi^\star}}^2\right] \\
&\stepc{=} \Omega\left(\frac{d\log d}{n} + \sqrt{\frac{d^3\log d}{mn}}\right), 
\end{align*}
where (a) is the scale-invariance of the loss $d_{\text{\rm NF}, \Sigma}$, i.e. $d_{\text{\rm NF}, \Sigma}(\pi^\star, \widehat{\pi}) = d_{\text{\rm NF}, \lambda\Sigma}(\pi^\star, \widehat{\pi})$ for every $\lambda>0$; (b) uses that $\Fnorm{\Sigma^{\widehat{\pi}} - \Sigma^{\pi^\star}} \le d_{\text{\rm NF}, \lambda\Sigma}(\pi^\star, \widehat{\pi})$ if $\Opnorm{\Sigma}\le 1$; and (c) follows from Theorem \ref{thm:lower_bound}. The proof is complete. 

\subsection{Proof of Theorem \ref{thm:oracle}}
\subsubsection{Proof of upper bound}
Similar to the proof of the GW upper bound, we begin with the definitions of several good events. Let $\cG = \calE_1 \cap \calE_2 \cap \calE_3$ be the good event, where
\begin{align}
\calE_1 &:= \bigg\{ \jiao{\hS_Y, \hS_m - \hS_m^{\widehat{\pi}}} \ge  \jiao{\Sigma, \hS_m - \hS_m^{\widehat{\pi}}} \nonumber \\
& \qquad - c\sqrt{\frac{d\log d+\log(1/\delta)}{n}}\Fnorm{\Sigma^{1/2}(\hS_m - \hS_m^{\widehat{\pi}})\Sigma^{1/2}} \bigg\}, \label{eq:oracle_E1} \\
\calE_2 &:= \bigg\{ \jiao{\Sigma, \hS_m - \hS_m^{\widehat{\pi}}} \ge  \jiao{\Sigma, \Sigma - \Sigma^{\widehat{\pi}}} - \frac{d}{\sqrt{m}}\Fnorm{\Sigma - \Sigma^{\widehat{\pi}^{-1} }} \bigg\}, \label{eq:oracle_E2} \\
\calE_3 &:= \bigg\{ \Fnorm{\hS_m - \hS_m^{\widehat{\pi}}} \le \Fnorm{\Sigma - \Sigma^{\widehat{\pi}}} + \frac{2d}{\sqrt{m}} \bigg\}, \label{eq:oracle_E3}
\end{align}
with a large enough numerical constant $c>0$. We show that $\bP(\calG)\ge 1-\delta$ under the oracle model. Using Hanson--Wright and triangle inequalities, the statements for $\calE_1$ and $\calE_3$ follow from the same arguments in Lemma \ref{lemma:GW_good_event} for their counterparts in \eqref{eq:GW_E1} and \eqref{eq:GW_E3}. As for $\calE_2$, we have
\begin{align*}
\jiao{\Sigma, \hS_m - \hS_m^{\widehat{\pi}}} &= \jiao{\Sigma - \Sigma^{\widehat{\pi}^{-1}}, \hS_m} \\
&\stepa{\ge} \jiao{\Sigma - \Sigma^{\widehat{\pi}^{-1}}, \Sigma} - \Fnorm{\Sigma - \Sigma^{\widehat{\pi}^{-1}}}\Fnorm{\hS_m - \Sigma} \\
&\stepb{\ge}  \jiao{\Sigma - \Sigma^{\widehat{\pi}^{-1}}, \Sigma} - \frac{d}{\sqrt{m}} \Fnorm{\Sigma - \Sigma^{\widehat{\pi}^{-1}}} \\
&= \jiao{\Sigma, \Sigma - \Sigma^{\widehat{\pi}}} - \frac{d}{\sqrt{m}} \Fnorm{\Sigma - \Sigma^{\widehat{\pi}^{-1}}}, 
\end{align*}
where (a) is the triangle inequality $\jiao{A,B}\le \|A\|_{\text{F}}\|B\|_{\text{F}}$, and (b) follows from the oracle property of $\hS_m$. This shows that $\calE_2$ holds almost surely. 

Now using exactly the same arguments in Section \ref{subsec:GW_loss_difference}, with \eqref{eq:oracle_E1}--\eqref{eq:oracle_E3} in place of \eqref{eq:GW_E1}--\eqref{eq:GW_E3}, for the GW estimator $\widehat{\pi}$ we have
\begin{align*}
    0 \ge \Fnorm{\Sigma^{\widehat{\pi}} - \Sigma}^2 - u\Fnorm{\Sigma^{\widehat{\pi}} - \Sigma} - v 
\end{align*}
under $\calG$, where
\begin{align*}
    u = 2c\sqrt{\frac{d\log d+\log(1/\delta)}{n}} + \frac{2d}{\sqrt{m}}, \qquad v = 2c\sqrt{\frac{d^2(d\log d+\log(1/\delta))}{mn}}. 
\end{align*}
By Lemma \ref{lemma:quad_bound}, this implies
\begin{align*}
\Fnorm{\Sigma^{\widehat{\pi}} - \Sigma}^2 = O(u^2+v) = O\left(\frac{d\log d+\log(1/\delta)}{n}+\frac{d^2}{m} + \sqrt{\frac{d^2(d\log d+\log(1/\delta))}{mn}}\right)
\end{align*}
holds with probability at least $1-\delta$. Integrating the tail over $\delta\in (0,1)$ gives the expectation upper bound
\begin{align*}
\bE \Fnorm{\Sigma^{\widehat{\pi}} - \Sigma}^2 = O\left(\frac{d\log d}{n} + \frac{d^2}{m} + \sqrt{\frac{d^3\log d}{mn}} \right) = O\left(\frac{d\log d}{n} + \frac{d^2}{m}\right), 
\end{align*}
where the last inequality is due to AM-GM. This proves the claimed upper bound in Theorem \ref{thm:oracle}.  
\subsubsection{Proof of lower bound}
We use the same lower bound program in Section \ref{sec:lower_bound}, with an additional specification for the oracle $\hS_m$. Recall that Section \ref{sec:lower_bound} shows the minimax lower bound
\begin{align*}
    \inf_{\widehat{\pi}}\sup_{\pi^\star, \Sigma: \Opnorm{\Sigma}\le 1} \bE_{\Sigma} \Fnorm{\Sigma^{\widehat{\pi}} - \Sigma^{\pi^\star}}^2 \gtrsim \frac{d\log d}{n}
\end{align*}
even if $m=\infty$ (i.e. known covariance $\Sigma$), the same lower bound also holds under the oracle setting by taking $\hS_m = \Sigma$. Therefore, it remains to prove that if $m\ge d$ under the oracle setting, we have
\begin{align*}
    \inf_{\widehat{\pi}}\sup_{\pi^\star, \Sigma: \Opnorm{\Sigma}\le 1} \bE_{\Sigma} \Fnorm{\Sigma^{\widehat{\pi}} - \Sigma^{\pi^\star}}^2 \gtrsim \frac{d^2}{m}. 
\end{align*}

We choose the following joint distribution of $(\pi^\star, \Sigma)$: set $\pi^\star\sim \text{Unif}(S_d)$, and conditioned on $\pi^\star$, set $\Sigma = (I_d+\eta S^{(\pi^\star)^{-1}})/2$ for $S\sim \text{Unif}(\calS_0)$, where the set of matrices $\calS_0$ is constructed in Lemma \ref{lemma:S0_construction}. We choose (recall that $m\ge d$)
\begin{align}\label{eq:eta_oracle}
 \eta = \frac{1}{2c_1\sqrt{m}} \le \frac{1}{2c_1\sqrt{d}}, 
\end{align}
and note that the loss function in \eqref{eq:loss_function} becomes
\begin{align*}
L(\pi^\star, \widehat{\pi}) &= \min_{\Sigma = (I_d+\eta S^{(\pi^\star)^{-1}})/2: S\in \calS_0} \Fnorm{\Sigma^{\pi^\star} - \Sigma^{\widehat{\pi}}}^2 \\
&= \frac{\eta^2}{4}\min_{S\in \calS_0} \Fnorm{S - S^{ (\pi^\star)^{-1} \circ \widehat{\pi} } }^2 = \frac{\eta^2}{4}\min_{S\in \calS_0} \Fnorm{S^{ (\pi^\star)^{-1} } - S^{\widehat{\pi}^{-1}}}^2. 
\end{align*}
As $\sum_{i=1}^d \mathbbm{1}(\pi_1^{-1}(i) \neq \pi_2^{-1}(i)) = \sum_{i=1}^d \mathbbm{1}(\pi_2\circ \pi_1^{-1}(i) \neq i) = \sum_{j=1}^d \mathbbm{1}(\pi_2(j) \neq \pi_1(j))$ for every $\pi_1, \pi_2\in S_d$, the same argument in Lemma \ref{lemma:p_Delta} gives $p_\Delta \le d^{-c_3d}$ for the quantity $p_\Delta$ in Lemma \ref{lemma:Fano}. Therefore,
\begin{align}\label{eq:Fano_oracle}
\bE\left[ \Fnorm{\Sigma^{\pi^\star} - \Sigma^{\widehat{\pi}}}^2 \right] \ge \frac{c_2\eta^2 d^2}{4}\left(1 - \frac{I(\pi^\star; \hS_m, Y^n) +\log 2}{c_3d\log d}\right). 
\end{align}

Next we specify the choice of $\hS_m$ and upper bound the mutual information. The choice of the oracle is simple: we take $\hS_m \equiv I_d/2$. Note that
\begin{align*}
\Opnorm{\hS_m - \Sigma} = \frac{\eta}{2}\Opnorm{S} \stepc{\le} \frac{\eta}{2}\cdot c_1\sqrt{d} \overset{\eqref{eq:eta_oracle}}{\le} \sqrt{\frac{d}{m}}, 
\end{align*}
where (c) follows from the second property of $\calS_0$. This shows that the oracle property is satisfied. 

Finally we show that $I(\pi^\star; \hS_m, Y^n) = 0$, therefore the lower bound of Theorem \ref{thm:oracle} follows from \eqref{eq:eta_oracle} and \eqref{eq:Fano_oracle}. To see it, note that $\hS_m$ is deterministic, and $\Sigma^{\pi^\star} = (I_d + \eta S)/2$ with $S\sim \text{Unif}(\calS_0)$, for any realization of $\pi^\star$. Since $Y^n\sim \calN(0,\Sigma^{\pi^\star})^{\otimes n}$, this shows that $Y^n$ and $\pi^\star$ are independent. It is consequently clear that $I(\pi^\star; \hS_m, Y^n) = 0$. 

\bibliographystyle{alpha}
\bibliography{bibliography.bib}
\end{document}